\documentclass[11pt]{article}
\usepackage{graphicx, enumitem}
\usepackage{amsmath,amssymb,amsthm, mathrsfs, mathtools}
\usepackage[usenames,dvipsnames]{xcolor}
\usepackage[backref]{hyperref}
\usepackage{xcolor}
\hypersetup{
    colorlinks,
    linkcolor={red!60!black},
    citecolor={green!60!black},
    urlcolor={blue!60!black}
}
\usepackage[T1]{fontenc}
\usepackage{lmodern}
\usepackage[babel]{microtype}
\usepackage[english]{babel}
\usepackage{comment}

\usepackage{geometry}
\numberwithin{equation}{section}
\numberwithin{figure}{section}

\usepackage{enumitem}

\theoremstyle{plain}
\newtheorem{thm}{Theorem}[section]
\newtheorem{theorem}[thm]{Theorem}
\newtheorem{prop}[thm]{Proposition}

\newtheorem{clm}[thm]{Claim}
\newtheorem{cor}[thm]{Corollary}
\newtheorem{conj}[thm]{Conjecture}
\newtheorem{lemma}[thm]{Lemma}
\newtheorem{question}[thm]{Question}

\theoremstyle{definition}
\newtheorem{definition}[thm]{Definition}

\theoremstyle{remark}
\newtheorem{rem}[thm]{Remark}

\renewcommand{\Pr}{\mathbb{P}}

\newcommand{\Hminusu}{H'}

\def\itm#1{\rm ({#1})}

\def\itmarab#1{\mbox{\itm{{\it #1\,}\arabic{*}\hspace{.05em}}}}

\newcommand{\eps}{\varepsilon}

\DeclarePairedDelimiter{\parens}{(}{)}
\DeclarePairedDelimiter{\set}{\{}{\}}
\DeclarePairedDelimiter{\brackets}{[}{]}
\DeclarePairedDelimiter{\card}{|}{|}

\title{Ramsey simplicity of random graphs}

\author{Simona Boyadzhiyska\thanks{Institut f\"ur Mathematik, Freie Universit\"at Berlin, Berlin, Germany. E-mails: \texttt{s.boyadzhiyska@fu-berlin.de}, \texttt{shagnik@mi.fu-berlin.de}}
\and
Dennis Clemens\thanks{Hamburg University of Technology, Institute of Mathematics, Hamburg, Germany. E-mails: \texttt{dennis.clemens@tuhh.de}, \texttt{pranshu.gupta@tuhh.de}}
\and
Shagnik Das\footnotemark[1]
\and
Pranshu Gupta\footnotemark[2]
}

\begin{document}
\maketitle 

\begin{abstract}
A graph $G$ is \emph{$q$-Ramsey} for another graph $H$ if in any $q$-edge-colouring of $G$ there is a monochromatic copy of $H$, and the classic Ramsey problem asks for the minimum number of vertices in such a graph. This was broadened in the seminal work of Burr, Erd\H{o}s, and Lov\'asz to the investigation of other extremal parameters of Ramsey graphs, including the minimum degree.

It is not hard to see that if $G$ is minimally $q$-Ramsey for $H$ we must have $\delta(G) \ge q(\delta(H) - 1) + 1$, and we say that a graph $H$ is \emph{$q$-Ramsey simple} if this bound can be attained. Grinshpun showed that this is typical of rather sparse graphs, proving that the random graph $G(n,p)$ is almost surely $2$-Ramsey simple when $\frac{\log n}{n} \ll p \ll n^{-2/3}$. In this paper, we explore this question further, asking for which pairs $p = p(n)$ and $q = q(n,p)$ we can expect $G(n,p)$ to be $q$-Ramsey
simple. We resolve the problem for a wide range of values of $p$ and $q$; in particular, we uncover some interesting behaviour when $n^{-2/3} \ll p \ll n^{-1/2}$.

\

\noindent\textbf{Keywords:} Ramsey Theory, random graphs, minimum degree
\end{abstract}

\section{Introduction}\label{sec:intro}
\subsection{Minimum degrees of minimal Ramsey graphs}
We say that a graph $G$ is \emph{$q$-Ramsey} for another graph $H$, and write $G\to_q H$, if, for any $q$-colouring of the edges of $G$, there exists a \emph{monochromatic} copy of $H$, that is, a copy of $H$ whose edges all have the same colour. The fundamental theorem of Ramsey~\cite{ramsey1930formallogic} asserts that $K_n \rightarrow_q H$ for sufficiently large $n$, and hence at least one such graph $G$ exists for any choice of $H$ and $q$. It is then natural to investigate the nature of graphs $G$ that are $q$-Ramsey for a given graph $H$. As a first step in this direction, we can ask how large such a graph $G$ needs to be, leading us to the definition of the most well-studied concept related to Ramsey graphs, the Ramsey number. In this language, the \emph{$q$-colour Ramsey number} of a graph $H$, denoted $r_q(H)$, is defined as the minimum number of vertices in a graph that is $q$-Ramsey for $H$. 
Over the past few decades, this parameter has been studied extensively for various choices of the graph $H$. Arguably the most important case is when $H$ is a complete graph. It was shown by Erd\H{o}s~\cite{erdos1947some} and Erd\H{o}s and Szekeres~\cite{erdos1935combinatorial} that $r_2(K_t)$ is exponential in $t$; more precisely, they proved $2^{t/2}\leq r_2(K_t) \leq 2^{2t}$. Despite considerable effort over the past eighty years, these remain essentially the best known bounds, with improvements only in the lower-order terms; the current best lower bound is due to Spencer~\cite{spencer1975ramsey},
while the best upper bound was recently announced by Sah~\cite{sah2020diagonal} (improving an earlier bound of Conlon~\cite{conlon2009new}).

\medskip

In the 1970s, researchers began exploring other properties of Ramsey graphs, and we shall be interested in the minimum degree, the study of which began with the paper of Burr, Erd\H{o}s, and Lov\'asz~\cite{burr1976graphs}. Of course, since any supergraph of a $q$-Ramsey graph for $H$ is itself $q$-Ramsey for $H$, the question of determining the smallest possible minimum degree among \emph{all} $q$-Ramsey graphs for $H$ is rather uninteresting: we can always add an isolated vertex and make the minimum degree zero. To avoid such trivialities, we restrict our attention to the subcollection $\mathcal{M}_q(H)$ of minimal Ramsey graphs. We say that $G$ is a \emph{minimal $q$-Ramsey} graph for $H$ if $G$ is $q$-Ramsey for $H$ and contains no proper subgraph with this property. In other words, removing any edge or vertex destroys the Ramsey property of the graph. We can then define the parameter $s_q(H)$, introduced in~\cite{burr1976graphs} for $q=2$, as the smallest minimum degree among all minimal $q$-Ramsey graphs for $H$; that is,
\begin{align*}
    s_q(H) = \min \set{\delta(G): G\in \mathcal{M}_q(H)},
\end{align*}
where as usual $\delta(G)$ denotes the minimum degree of $G$.  

\medskip

When studying this parameter, there are a couple of easy general bounds one can give. For an upper bound, observe that since, by definition, $K_{r_q(H)} \rightarrow_q H$, any minimal $q$-Ramsey subgraph of this complete graph bears witness to the fact that $s_q(H) \le r_q(H) - 1$. From below, as observed by Fox and Lin~\cite{fox2007minimum}, a simple argument using the pigeonhole principle shows $s_q(H) \ge q(\delta(H) - 1) + 1$. Note that these bounds are typically very far apart: when $H = K_t$, for instance, the lower bound is linear in $t$ while the upper bound is exponential.

\medskip

In the original paper of Burr, Erd\H{o}s, and Lov\'asz~\cite{burr1976graphs}, the authors showed that $s_2(K_t) = (t-1)^2$, a surprising result for two reasons. First, while the two-colour Ramsey number of $K_t$ is still unknown for any $t \ge 5$, we can determine $s_2(K_t)$ precisely. Second, $s_2(K_t)$ is significantly smaller than $r_2(K_t)$. Informally, this means that a large Ramsey graph for $K_t$ can have a vertex of very low degree whose removal destroys the Ramsey property. 

Since its introduction in~\cite{burr1976graphs}, the parameter $s_q(H)$ has been studied for a number of different choices of $H$ and for larger $q$; see, for example,~\cite{bbl2020,BCG2020a,fox2014ramsey,fox2016minimum,fox2007minimum,grinshpun2017minimum, grinshpun2015some,gw2020triangle,hanrodlszabo2018,szabo2010minimum}. To the best of our knowledge, in all cases studied the value of $s_q(H)$ is far away from the trivial upper bound. On the other hand, the lower bound of Fox and Lin~\cite{fox2007minimum} has been shown to be tight for many graphs. Following Grinshpun~\cite{grinshpun2015some}, we call such a graph \emph{$q$-Ramsey simple}.

\begin{definition}\label{def:Ramsey-simple}\rm
A graph $H$ without isolated vertices is said to be \emph{$q$-Ramsey simple} if
\begin{align*}
    s_q(H) = q(\delta(H) - 1)+1.
\end{align*}
If $H$ has isolated vertices, then we say that $H$ is $q$-Ramsey simple if the graph obtained from $H$ by removing all isolated vertices is $q$-Ramsey simple. 
\end{definition}

Observe that adding isolated vertices to a graph does not affect the structure of the corresponding Ramsey graphs significantly. Indeed, if $H$ is a graph without isolated vertices and $H + tK_1$ is the graph obtained from $H$ by adding $t \ge 0$ isolated vertices, it is not difficult to check that $G \in \mathcal{M}_q(H)$ if and only if $G + sK_1 \in \mathcal{M}_q(H + tK_1)$, where $s = \max \set{0, t- (v(G) - v(H))}$.

\medskip

Previous work by Fox and Lin~\cite{fox2007minimum}, Szab\'o, Zumstein, and Z\"urcher~\cite{szabo2010minimum}, and Grinshpun~\cite{grinshpun2015some} has established the $2$-Ramsey simplicity of a wide range of bipartite graphs. Further results were proven in~\cite{BCG2020a}, including the $q$-Ramsey simplicity of all cycles of length at least four, for any number of colours $q \ge 2$. Based on these results, it is believed that simplicity is a more widespread phenomenon.

\begin{conj}[Szab\'o, Zumstein, and Z\"urcher~\cite{szabo2010minimum}] \label{conj:bipartite}
Every bipartite graph is $2$-Ramsey simple.
\end{conj}

The conjecture suggests that Ramsey simplicity is quite common, but it is natural to wonder whether this extends beyond the bipartite setting, given that we know cliques are not simple. Are cliques an exceptional case, or is $q$-Ramsey simplicity atypical for non-bipartite graphs? In somewhat more precise terms, when can we expect the $n$-vertex binomial random graph $G(n,p)$, where every edge appears independently with probability $p$, to be $q$-Ramsey simple?

\subsection{Random graphs} \label{intro:random.graphs}
Random graphs have long played an important role in Ramsey Theory: Erd\H{o}s's famous exponential lower bound on the Ramsey numbers of complete graphs in~\cite{erdos1947some} came from analysing the clique and independence numbers of random graphs, while a key ingredient in the best modern upper bounds is showing that large Ramsey graphs must be random-like. When it comes to more general Ramsey problems, the work of R\"odl and Ruci\'nski~\cite{rodl1993lower,rodl1995threshold} establishes, for a given graph $H$ and number of colours $q$, the range of values of $p$ for which we have $G(n,p) \rightarrow_q H$ with high probability.

In these seminal papers, which have inspired a great deal of subsequent research, the random graph plays the role of the host graph $G$, while the target graph $H$ is fixed in advance. Surprisingly, there has been considerably less work in the setting where the target graph $H$ is itself random. When $H \sim G(n,p)$, Fox and Sudakov~\cite{fox2009two} and Conlon~\cite{conlon2013ramsey} provide some lower and upper bounds on $r_2(H)$ for different ranges of $p$, while Conlon, Fox, and Sudakov~\cite{conlon2020short} show that $\log r_2(H)$ is well-concentrated.

\medskip

In this paper we shall focus on the minimum degree of Ramsey graphs for the random graph $G(n,p)$, with the goal of determining when it is $q$-Ramsey simple. This line of research was initiated by Grinshpun~\cite{grinshpun2015some}, who proved that sparse random graphs are 2-Ramsey simple with high probability. 

\begin{theorem}[Corollary 2.1.4~in~\cite{grinshpun2015some}]\label{thm:foxetal}
Let $p=p(n)\in(0,1)$ and $H\sim G(n,p)$. If \mbox{$\frac{\log n}{n}\ll p \ll n^{-2/3}$}, then a.a.s.\ $H$ is $2$-Ramsey simple.
\end{theorem}

In this range of edge probabilities the random graph is almost surely not bipartite (in fact, its chromatic number is unbounded), showing that Conjecture~\ref{conj:bipartite} does not tell the full story. Moreover, the argument in~\cite{grinshpun2015some} can easily be extended to provide, for any fixed $q \in \mathbb{N}$, $q$-Ramsey simplicity for $G(n,p)$ in the above range of $p$. This begs two natural questions: what happens when the number of colours $q$ grows with $n$, and what happens in other ranges of the edge probability $p$?

\subsection{Results}

In this paper, we settle this question for a wide range of parameters, but in order to present our results, we need to introduce some notation.

\medskip

We first remark that the parameter $s_q(H)$ and the notion of Ramsey simplicity are not monotone in the graph $H$. As we shall observe in Section~\ref{sec:monotone}, a Ramsey simple graph can have both subgraphs and supergraphs that are themselves not Ramsey simple, while a graph that is not Ramsey simple can have simple subgraphs and supergraphs. However, we do have monotonicity in the number of colours $q$, and we shall demonstrate that $(q+1)$-Ramsey simplicity implies $q$-Ramsey simplicity. Hence, we can ask for a threshold value for $q$, i.e., the largest number $\tilde{q}$ of colours for which a given graph is $\tilde{q}$-Ramsey simple. We set
\[
 \tilde{q}(H) :=
 \sup\{q: \; H \text{ is $q$-Ramsey simple} \}.
\]
Note that every graph is, by definition, $1$-Ramsey simple, since the only minimal $1$-Ramsey graph for $H$ is $H$ itself, and so $s_1(H) = \delta(H)$. Thus, when a graph $H$ is not $q$-Ramsey simple for any number of colours $q \ge 2$, we have $\tilde{q}(H) = 1$. At the other extreme, if $H$ is $q$-Ramsey simple for any number of colours $q$, we have $\tilde{q}(H) = \infty$.

\medskip

Given this notation, we can now state our main theorem, which collects various bounds we were able to prove for the threshold $\tilde{q}(H)$ when $H \sim G(n,p)$.

\begin{theorem}\label{thm:main}
Let $p = p(n) \in (0,1)$ and $H\sim G(n,p)$. Let $u \in V(H)$ be a vertex of minimum degree $\delta(H)$ and let $F = H[N(u)]$ be the subgraph of $H$ induced by the neighbourhood of $u$. Denote by $\lambda(F)$ the order of the largest connected component in $F$. Then a.a.s.\ the following bounds hold:
\begin{enumerate}[leftmargin=40pt,label = (\alph*)]
    \item\label{thm:main:forest} $\tilde{q}(H)=\infty$ 
       \hfill if $0<p\ll n^{-1}$.
    \item\label{thm:main:alwayssimple} $\tilde{q}(H)=\infty$ 
       \hfill if $\frac{\log n}{n}\ll p \ll n^{-\frac{2}{3}}$.
    \item\label{thm:main:intermediatelowerbounds} $\tilde{q}(H) \ge (1 + o(1)) \max \set*{ \frac{\delta(H)}{\lambda(F)^2}, \frac{\delta(H)}{80 \log n} }$ \hfill if $n^{-\frac23} \ll p \ll n^{-\frac12}$.
    \item\label{thm:main:intermediateupperbounds} $\tilde{q}(H) \le (1 + o(1)) \min \set*{ \frac{\delta(H)}{\Delta(F)}, \frac{\delta(H)^2}{2 e(F)}}$ \hfill if $n^{-\frac23} \ll p \ll 1$.
    \item\label{thm:main:neversimple} $\tilde{q}(H)=1$ 
        \hfill if $\left(\frac{\log n}{n}\right)^{1/2}\ll p < 1$.
\end{enumerate}
\end{theorem}

As shown above, we extend Theorem~\ref{thm:foxetal} by showing that these sparse random graphs are not just $q$-Ramsey simple for any fixed $q$, but even when the number of colours $q$ is allowed to grow with $n$. On the other hand, we prove that much denser random graphs are not simple for any number of colours $q \ge 2$. Thus, both extremes are observed for different edge probabilities. Most interestingly, though, the simplicity threshold for random graphs of intermediate density depends on some parameters of the random graph itself --- these graphs are $q$-Ramsey simple for small values of $q$, but not when $q$ grows too large.

\begin{rem}
As suggested by the above bounds, this dependence on $q$ is governed by the subgraph $F$, and it is the appearance of edges in $F$ that gives rise to a finite bound on $\tilde{q}(H)$. When $p \ll n^{-\frac23}$, then $F$ almost surely has no edges, while if $p \gg n^{-\frac23}$, then $F$ almost surely does, explaining the distinction between cases~\ref{thm:main:alwayssimple} and~\ref{thm:main:intermediateupperbounds}. When $p = \Theta \parens*{n^{-\frac23}}$, then $F$ is empty (and $\tilde{q}(H)$ thus infinite) with probability bounded away from $0$ and $1$.
\end{rem}

When $n^{-\frac23} \ll p \ll \parens*{\frac{\log n}{n}}^{\frac12}$, by analysing the structure of random graphs, we can give quantitative estimates for the bounds on $\tilde{q}(H)$ in this intermediate range.

\begin{cor}\label{cor:bounds}
Let $k \ge 2$ be a fixed integer and let $f = f(n)$ satisfy $1 \ll f = n^{o(1)}$. Let $p = p(n)$ satisfy $n^{-\frac23} \ll p \ll \left( \frac{\log n}{n} \right)^{\frac12}$ and let $H \sim G(n,p)$. Then a.a.s.\ the following bounds hold:
\begin{enumerate}[leftmargin=40pt,label = (\alph*)]
    \item\label{cor:bounds:constanttreecomponents} if $n^{-\frac{k}{2k-1}}\ll p \ll n^{-\frac{k+1}{2k+1}}$, then $(1+o(1)) \frac{np}{k^2} \le \tilde{q}(H) \le (1+o(1))\frac{np}{k-1}$.
    \item \label{cor:bounds:boundary} if $p = \Theta\parens*{ n^{-\frac{k+1}{2k+1}} }$, then $(1 + o(1)) \frac{np}{(k+1)^2} \le \tilde{q}(H) \le (1 + o(1))\frac{np}{k-1}.$
    \item\label{cor:bounds:smalldegrees} if $p = n^{-\frac12} f^{-1}$, then $(1+ o(1))\frac{np}{\log n} \max \set*{\frac{16 \log^2 f}{\log n}, \frac{1}{80}} \le \tilde{q}(H) \le (2+ o(1)) \frac{np \log ( f^2 \log n)}{\log n}$.
    \item\label{cor:bounds:manyedges} if $n^{-\frac12} \ll p \ll \left( \frac{\log n}{n} \right)^{\frac12}$, then $1 \le \tilde{q}(H) \le (8+ o(1))\frac{1}{p}$.
\end{enumerate}
\end{cor}

Corollary~\ref{cor:bounds} shows that, for fixed $\varepsilon > 0$ and $n^{-\frac23} \ll p \ll n^{-\frac12 - \varepsilon}$, we determine the threshold up to a constant factor, while for $n^{-\frac12 - o(1)}$, we know it up to a polylogarithmic factor. Most surprisingly, these bounds reveal that the threshold $\tilde{q}(H)$ evolves in a complicated fashion: while it drops from $\infty$ to $1$ as $p$ ranges from $\frac{\log n}{n}$ to $1$, it does not do so in a monotone fashion, as it must increase in the ranges $p \in \left( n^{-\frac{k}{2k-1}}, n^{-\frac{k+1}{2k+1}} \right)$ for each fixed $k$. These results are illustrated in Figure~\ref{fig:simplicity-graph}.

\begin{figure}[!ht]
    \centering
    \includegraphics[scale=0.8,page=2]{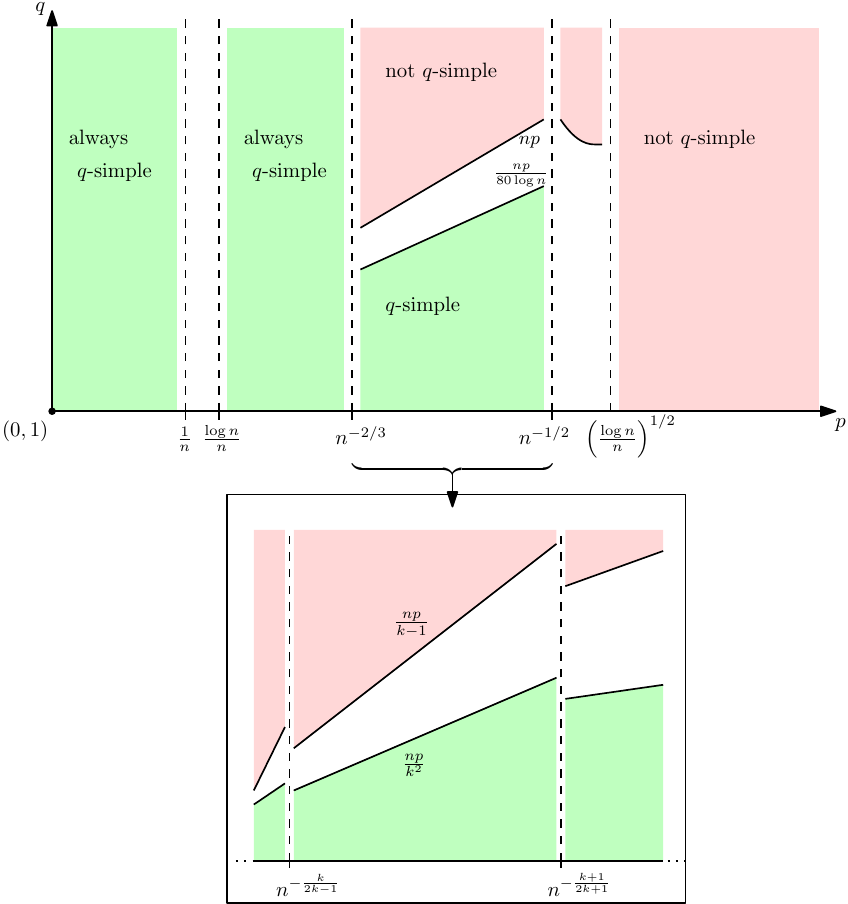}
    \caption{Bounds on the simplicity threshold $\tilde{q}(G(n,p))$}
    \label{fig:simplicity-graph}
\end{figure}

\subsection{Organisation of the paper}
In Section~\ref{sec:monotone} we discuss the monotonicity of Ramsey simplicity, and shall in particular justify the definition of the threshold $\tilde{q}(H)$. We then turn to random graphs, and in Section~\ref{sec:gnpproperties} collect properties of random graphs needed to derive Corollary~\ref{cor:bounds} from Theorem~\ref{thm:main}. The constructions of Ramsey graphs that establish the lower bounds in Theorem~\ref{thm:main} are provided in Section~\ref{sec:simplicity}, while the upper bounds on $\tilde{q}(H)$ are proven in Section~\ref{sec:non-simplicity}. The final section, Section~\ref{sec:concluding}, is devoted to concluding remarks and open problems.

\subsection{Notation}
The notation used in this paper is mostly standard, except that for a graph $G$, we write $\lambda(G)$ for the order of the largest connected component in $G$. Throughout the paper, a $q$-colouring is an edge-colouring of a given graph with $q$ colours and, unless otherwise specified, we will take $[q]$ to be our colour palette.

\section{Monotonicity in $q$}\label{sec:monotone}
In this section we will prove that the property of being $q$-Ramsey simple is monotone decreasing in the number of colours; that is, we will show that if a graph is not $q$-Ramsey simple for some $q$, then it cannot be $q'$-Ramsey simple for any $q'\geq q$.

\begin{lemma}\label{lem:qmonotonicity}
If $H$ is not $q$-Ramsey simple, then $H$ is not $(q+1)$-Ramsey simple.
\end{lemma}

Note that $q$-Ramsey simplicity does not observe any monotonicity with respect to the graph $H$. Indeed, we know that any tree on $t$ vertices is $2$-Ramsey simple, whereas the clique $K_t$ is not. Similarly, there exist graphs that are $q$-Ramsey simple but contain subgraphs that are not. For instance, Theorem~2.1.3 in~\cite{grinshpun2015some}  shows that any 3-connected graph $H$ containing a vertex $v$ of minimum degree such that $N(v)$ is contained in an independent set of size $2\delta(H)-1$ is 2-Ramsey simple. 
Hence, while $K_{\delta}$ for $\delta\geq 3$ is not $2$-Ramsey simple,
the following supergraph of it is: add $2\delta-1$ new vertices to $K_\delta$ with a complete bipartite graph connecting them to the clique, and then add another vertex $v$ connected to exactly $\delta$ of the  $2\delta-1$ new vertices.

\begin{proof}[Proof of Lemma~\ref{lem:qmonotonicity}]
Assume $H$ is not $q$-Ramsey simple, that is, $s_q(H) > q(\delta(H)-1)+1$. 
Suppose for a contradiction that there exists a graph $G\in \mathcal{M}_{q+1}(H)$ such that $G$ contains a vertex $v$ of degree $(q+1)(\delta(H)-1)+1$. Let $e$ be an arbitrary edge incident to $v$.

By the minimality of $G$, we know that the graph $G-e$ has an $H$-free $(q+1)$-colouring $c$. Now, if there are at most $\delta(H)-2$ edges that are incident to $v$ and have colour $q+1$ under $c$, then we can give $e$ colour $q+1$ to obtain an $H$-free $(q+1)$-colouring of $G$, contradicting $G \in \mathcal{M}_{q+1}(H)$. Hence we may assume that there are at least $\delta(H)-1$ edges incident to $v$ that have colour $q+1$. Let $G_0$ be the subgraph of $G$ containing all edges that have colours in $[q]$ under $c$ together with the edge $e$, i.e., $G_0 = G - c^{-1}(q+1)$. We then know that $d_{G_0}(v) \leq (q+1)(\delta(H)-1)+1 - (\delta(H)-1) = q(\delta(H)-1)+1 < s_q(H)$. If $G_0$ is not $q$-Ramsey for $H$, then $G_0$ has an $H$-free $q$-colouring $c'$, and extending $c'$ to the graph $G$ by colouring the edges in $E(G)\setminus E(G_0)$ with colour $q+1$ gives an $H$-free $(q+1)$-colouring of $G$, a contradiction. Therefore, $G_0\rightarrow_q H$. 

But $d_{G_0}(v) < s_q(H)$, so $G_0$ cannot be minimal $q$-Ramsey for $H$, and in particular, the vertex $v$ cannot be part of a minimal $q$-Ramsey subgraph of $G_0$. Thus $G_0-v\rightarrow_q H$. But the restriction of $c$ to $G_0-v$ is $H$-free by our choice of $c$, which again leads to a contradiction.

Hence, $H$ cannot be $(q+1)$-Ramsey simple.
\end{proof}

\section{Properties of $G(n,p)$} \label{sec:gnpproperties}

In this section we shall establish various properties of the random graph $G(n,p)$ needed for the proof of Theorem~\ref{thm:main} and the deduction of Corollary~\ref{cor:bounds}.

\subsection{Facts about $G(n,p)$}\label{sec:gnpfacts}

We start with some bounds on the degrees and edge distribution in the random graph, for which we require the following well-known concentration bounds due to Chernoff (see~\cite[Theorem 2.3]{mcdiarmid1998concentration} and~\cite[Theorem 22.6]{frieze2016introduction}).

\begin{lemma} \label{lem:Chernoff}
Let $X \sim Bin(n,p)$ and $\mu = \mathbb{E}[X]$.
\begin{enumerate}[leftmargin=40pt,label = (\alph*)]
    \item If $0<\varepsilon<1$, then $\mathbb{P}(X\geq (1+\varepsilon)\mu)\leq \exp{(\frac{-\mu \varepsilon^2}{3})}$ and $\mathbb{P}(X\leq (1-\varepsilon)\mu)\leq \exp{(\frac{-\mu \varepsilon^2}{2})}$. \label{Chernoff1}
    \item For all $t \geq 7\mu$ we have $\mathbb{P}(X \geq t) \leq \exp(-t)$. \label{Chernoff3}
\end{enumerate}
\end{lemma}

With these concentration results, we can specify how many edges the random graph is likely to have. This is done in the following lemmas, which collect some folklore bounds on the degrees and number of edges in $G(n,p)$. We start by controlling the degrees.

\begin{lemma}[Degrees in $G(n,p)$]\label{lem:gnp_degrees}
Let $p=p(n)\in (0,1)$, and let $H\sim G(n,p)$. Then a.a.s.~the following bounds on the maximum degree hold:
\begin{enumerate}[leftmargin=40pt,label = (\alph*)]
    \item\label{lem:gnp_degree_max1} for any fixed integer $k \ge 2$, we have $\Delta(H)\geq k-1$ when $p\gg n^{-\frac{k}{k-1}}$, and
    \item\label{lem:gnp_degree_max2} for any $f = f(n)$ satisfying $1 \ll f = n^{o(1)}$, we have $\Delta(H) \geq \frac{\log n}{\log (f \log n)}$ when $p = \frac{1}{nf}$.
\end{enumerate}
Moreover, if $p\gg \frac{\log n}{n}$,
then with probability at least $1-n^{-2}$ we have
\begin{enumerate}[leftmargin=40pt, start=3, label = (\alph*)]
    \item\label{lem:gnp_degree_concentration} $d_H(v)=(1\pm o(1))np$ for every $v\in V(H)$.
\end{enumerate}
\end{lemma}

\begin{proof}
If $p\gg n^{-\frac{k}{k-1}}$ for any integer $k\geq 2$, then it follows from a simple second moment calculation that $H$ a.a.s.~contains a star with $k-1$ edges, and hence $\Delta(H) \ge k-1$; see Theorem~5.3 in~\cite{frieze2016introduction} for more details.

Part~\ref{lem:gnp_degree_max2} can be obtained from a similar application of the second moment method. For simplicity, we apply Theorem~3.1~(ii) from~\cite{bollobas2001random}, stating that, if $n^{-3/2} \ll p \ll 1-n^{-3/2}$ and the expected number of vertices of degree $d = d(n)$ in $H \sim G(n,p)$ tends to infinity, then with high probability $H$ contains at least one vertex of degree $d$. 

For $p = \frac{1}{nf}$ and 
$d = \frac{\log n}{\log (f \log n)}$, we can lower bound the expected number of vertices of degree $d$ by
\begin{align*}
    n \binom{n-1}{d} p^d (1-p)^{n-1-d} 
    & \ge n \left( \frac{(n-1)p}{d} \right)^d (1 - np) \ge \tfrac12 n \left( \frac{1}{2fd} \right)^d \\
    & = \frac{1}{2}e^{\log n - d \log(2fd) }
    = \frac{1}{2}e^{ \frac{\log n}{\log (f \log n)} \left(\log\log \sqrt{f\log n}\right) }
\end{align*}
which tends to infinity. Thus we must have at least one vertex of degree $d$, and hence $\Delta(H) \ge d$.

Part (c) follows by applying Lemma~\ref{lem:Chernoff}\ref{Chernoff1} to the degree of each vertex, and then taking a union bound over all $n$ vertices.
\end{proof}

We can also bound the number of edges, both globally and, provided the edge probability is not too low, in all large induced subgraphs.

\begin{lemma}[Edge counts in $G(n,p)$]\label{lem:gnp_edge.counts}
Let $p=p(n)\in (0,1)$ with $p\gg n^{-2}$, and let $H\sim G(n,p)$. Then a.a.s.~the following statements hold:
\begin{enumerate}[leftmargin=40pt,label = (\alph*)]
    \item\label{lem:gnp_edge.counts_concentration1} $e(H)=(1\pm o(1))\frac{n^2p}{2}$, and
\end{enumerate}    
\begin{enumerate}[leftmargin=40pt, start=2, label = (\alph*)]
    \item\label{lem:gnp_edge.counts_concentration2} if $p \gg \frac{\log n}{n}$, then with probability at least $1 - n^{-2}$, every set $S \subseteq V(H)$ of size $s \ge \frac{20 \log n}{p}$ satisfies $e_H(S)\geq \frac{1}{4}s^2p$.
\end{enumerate}
\end{lemma}

\begin{proof}
Part~\ref{lem:gnp_edge.counts_concentration1} follows directly from Lemma~\ref{lem:Chernoff}\ref{Chernoff1}.
For part~\ref{lem:gnp_edge.counts_concentration2}, we take a union bound over all sets $S$ of $s$ vertices, again applying Lemma~\ref{lem:Chernoff}\ref{Chernoff1} to bound the probability that such a set contains too few edges. This results in a bound of 
\begin{align*}
    \sum_{s=\frac{20\log n}{p}}^n \binom{n}{s} e^{-\frac{1}{16}(1-o(1))s^2p}
    \leq \sum_{s=\frac{20\log n}{p}}^n e^{s\log n -
        \frac{1}{16}(1-o(1))s^2p} 
    \leq \sum_{s=\frac{20\log n}{p}}^n e^{-0.2 s\log n} 
    < n^{-2},
\end{align*}
proving the lemma.
\end{proof}

Aside from knowing how many edges the random graph contains, we shall also need some knowledge about how they are distributed. The following result describes the structure of sparse random graphs.

\begin{lemma}\label{lem:gnp_forest.regime}
Let $p=p(n)\in (0,1)$ with $p\ll n^{-1}$, and let $H\sim G(n,p)$. Then a.a.s. $H$ is a forest, and moreover the order $\lambda(H)$ of its largest component satisfies the following bounds:
\begin{enumerate}[leftmargin=40pt,label = (\alph*)]
    \item\label{lem:gnp_forest.regime_log} $\lambda(H) \le \log n$,
    \item\label{lem:gnp_forest.regime_constant} if $p\ll n^{-\frac{k+1}{k}}$ for some constant $k \in \mathbb{N}$, then $\lambda(H) \le k$, and
    \item\label{lem:gnp_forest.regime_intermediate} if $p = \frac{1}{nf}$ for some $f = f(n)$ satisfying $1 \ll f = n^{o(1)}$, then $\lambda(H) \le (1 + o(1)) \frac{\log n}{\log f}$.
\end{enumerate}
\end{lemma}

\begin{proof}
That $H$ contains no cycles, and hence is a forest, can be shown by taking a union bound over all possible cycles; see Theorem~2.1
in~\cite{frieze2016introduction} for the details. We now bound the orders of the trees in this forest. For part~\ref{lem:gnp_forest.regime_log}, we refer to Lemma~2.12(ii) in~\cite{frieze2016introduction}, which asserts that with high probability a random graph $H'\sim G(n,e^{-2} n^{-1})$ contains no trees of order larger than
$\log n$. By monotonicity the same bound holds when $p \ll n^{-1}$. 
For the bound in~\ref{lem:gnp_forest.regime_constant}, notice that
there are only a constant number of non-isomorphic trees on
$k+1$ vertices, and by a simple first moment calculation (see Theorem 5.3 in~\cite{frieze2016introduction}) each of these trees appears in $H$ with vanishing probability when $p \ll n^{-\frac{k+1}{k}}$. The bound in part~\ref{lem:gnp_forest.regime_intermediate} can again be obtained by running a first moment calculation, the details of which we now sketch. As there are $k^{k-2}$ labelled trees on $k$ vertices, the total possible number of tree components of order $k$ is $\binom{n}{k} k^{k-2}$. 
For such a tree to appear as a subgraph, we need its $k-1$ edges to appear in $G(n,p)$.
Hence, the probability of seeing such a tree is at most
\[ \binom{n}{k} k^{k-2} p^{k-1} \le \left( \frac{ne}{k} \right)^k \frac{(kp)^k}{k^2p}  
\le p^{-1} (nep)^k . \]
For $p = \frac{1}{nf}$ and any $\eps>0$, the sum of this expression over all $k \ge (1 + \eps) \frac{\log n}{\log f}$ is at most
\[
p^{-1} \left(nep\right)^{(1 + \eps) \frac{\log n}{\log f}} \left( \frac{1}{1 - nep} \right) \le 2 e^{\log (nf) - (1+\eps)\frac{\log n}{\log f} (\log f - 1)} \le 
2e^{-\frac{\eps}{2}\log n} = o(1)
\]
and hence a.a.s.~the largest component has order at most $(1 + o(1)) \frac{\log n}{\log f}$.
\end{proof}

Switching to a much denser range, we find that when the edge probability is sufficiently large, not only does $G(n,p)$ contain cycles, but every edge is contained in a triangle.

\begin{lemma}\label{lem:gnp_triangles}
Let $p = p(n)\in (0,1)$ be such that  $p \gg \sqrt{\frac{\log n}{n} }$, and let $H\sim G(n,p)$. 
Then a.a.s.~every edge of $H$ is contained in a triangle.
\end{lemma}

\begin{proof}
An easy application of the union bound gives
\begin{align*}
    & \Pr\left(\exists e=uv\in E(H):~ uv~ \text{is not in a triangle} \right) \\
    \leq~  &
    \Pr\left(
    \exists \{u,v\}\in\binom{V(H)}{2}:~ uw\notin E(H)~ \text{or}~vw\notin E(H)~ \text{for all }w\in V(H)\setminus\{u,v\}
    \right)\\
    \leq~ &
    \binom{n}{2}\cdot (1-p^2)^{n-2} < e^{2\log n - p^2(n-2)} = o(1)\, ,
\end{align*}
which proves the lemma.
\end{proof}

Finally, in our construction of minimal Ramsey graphs with vertices of low degree, we shall make use of some mild pseudorandom properties concerning the degrees, connectivity, and expansion of the target graph $H$. The required properties are collected in the definition below.

\begin{definition}[Well-behaved] \label{def:wellbehaved}
We say an $n$-vertex graph $H$ is \emph{well-behaved} if it satisfies the following properties:
\begin{enumerate}[label=\itmarab{W}]
    \item\label{def:wb_min_unique} $H$ has a unique vertex $u$ of minimum degree $\delta(H)$,
    \item\label{def:wb_pair_degree} every pair of vertices in $H$ has codegree at most $\tfrac12 \delta(H)$,
    \item\label{def:wb_3conn} $H$ is $3$-connected, and
    \item\label{def:wb_components} removing $\delta(H)$ vertices from $H$ cannot create a component of size $k \in \brackets*{\tfrac12 \delta(H), \tfrac12 n} $.
\end{enumerate}
\end{definition}

As might be expected, random graphs are highly likely to be well-behaved.

\begin{lemma} \label{lem:wellbehaved}
If $\frac{\log n}{n} \ll p \ll 1$ then a.a.s.~$H \sim G(n,p)$ is well-behaved.
\end{lemma}

\begin{proof}
The property~\ref{def:wb_min_unique} is established in Theorem~3.9(i) of~\cite{bollobas2001random}.
Moreover, by Lemma~\ref{lem:gnp_degrees}\ref{lem:gnp_degree_concentration} we may condition
on $\delta(H)=(1\pm o(1))np$ from now on.
For property~\ref{def:wb_pair_degree}, observe that the distribution of the codegree of a given pair of vertices is $\mathrm{Bin}(n-2,p^2)$. We consider two cases. If $p^2 \geq \frac{10 \log n}{n}$
then by applying a Chernoff bound (Lemma~\ref{lem:Chernoff}\ref{Chernoff1}) 
we obtain \mbox{$\Pr(d_H(u,v) \geq 2np^2) \leq e^{-\frac{1}{3}(n-2)p^2} < e^{-3\log n}$} for large $n$. Taking a union bound over all $\binom{n}{2}$ pairs of vertices, this shows that with high probability the maximum codegree is at most $2np^2 < \tfrac12 \delta(H)$. 
Otherwise $p^2 \leq \frac{10 \log n}{n}$ and then Lemma~\ref{lem:Chernoff}\ref{Chernoff3} yields $\Pr(d_H(u,v) \geq 100 \log n) \leq e^{-100 \log n}$. We can then again take a union bound over all pairs to show the maximum codegree is at most $100 \log n$, which, as $\delta(H) = (1\pm o(1)) np \gg \log n$, is again less than half the minimum degree.

Property~\ref{def:wb_3conn} is shown to hold with high probability in Theorem~4.3 in~\cite{frieze2016introduction}.

This leaves us with property~\ref{def:wb_components}. Let us fix $k \in \brackets*{\tfrac12 \delta(H), \tfrac12 n}$, and bound the probability that we can create a component $K$ of size $k$ by removing a set $U$ of $\delta(H)$ vertices. In order for this to happen, there cannot be any edges between $K$ and $V(H) \setminus \left( K \cup U \right)$. For given $K$ and $U$, the probability of this is $(1-p)^{k(n-k-\delta(H))}$. Taking a union bound over all possible components $K$ and cut-sets $U$, the probability that property~\ref{def:wb_components} fails for a given $k$ is at most
\[ \binom{n}{k} \binom{n}{\delta(H)} \left( 1 - p \right)^{k(n-k-\delta(H))} \le n^k n^{\delta(H)} e^{-pk(n-k-\delta(H))} \le e^{(k+ \delta(H)) \log n - \tfrac14 pkn}, \]
where the last inequality uses the bounds $k \le \tfrac12 n$ and $\delta(H) = (1\pm o(1)) np \le \tfrac14 n$. Now, since $p \gg \frac{\log n}{n}$, we have $k \log n \ll pkn$, and, since $k \ge \tfrac12 \delta(H)$, we also have $\delta(H) \log n \ll pkn$. Hence, we can bound this error probability by $e^{-\frac18 pkn} = o(n^{-1})$, using again the fact that $pkn \gg \delta(H) \log n$. Therefore, even after taking a union bound over all possible values of $k$, we see that property~\ref{def:wb_components} holds with high probability.
\end{proof}

\smallskip

\subsection{Transference lemma}

As is evident in the statement of Theorem~\ref{thm:main}, our bounds on the simplicity of $H \sim G(n,p)$ depend on the subgraph induced by the neighbourhood of the minimum degree vertex (which, by virtue of Lemma~\ref{lem:wellbehaved} and property~\ref{def:wb_min_unique}, we may assume to be unique). Our next lemma allows us to transfer what we know about the random graph $G(\delta(H), p)$ to this subgraph.

\begin{lemma}\label{lem:gnp_neighbourhood_general}
Let $p=p(n)\in (0,1)$ be such that $p\gg \frac{\log n}{n}$.
For every $s\in [0.5np,2np]$, let
$\mathcal{P}_s$ be a graph property, and assume that a random graph 
$G_s\sim G(s,p)$ satisfies
$$
\Pr\left( G_s \in \mathcal{P}_s \right) = 1 - o(1).
$$
Then $H\sim G(n,p)$ a.a.s.~has a unique minimum degree vertex $u$ and $H[N_H(u)]\in \mathcal{P}_{d_H(u)}$.
\end{lemma}

\begin{proof}
Let us fix some $\beta_n=o(1)$ such that
\begin{align}\label{beta}
\Pr\left( G_s \notin \mathcal{P}_s \right) = o(\beta_n)
\end{align}

for every $s\in [0.5np,2np]$.
Moreover, let $X_\delta$ denote the event that $H$ has a unique vertex 
of minimum degree and $0.5np\leq \delta(H) \leq 2np$.
By Lemma~\ref{lem:wellbehaved}, specifically property~\ref{def:wb_min_unique}, and Lemma~\ref{lem:gnp_degrees}
we know that $X_\delta$ holds with high probability.
In particular, we can find $\delta_n=o(1)$ such that 
$$
\Pr(0.5np\leq \delta(H) \leq 2np) \geq 
\Pr(X_{\delta})=1-\delta_n .$$   
In the following we will condition on the event $X_\delta$,
and whenever we do so, we will always let $u$ 
denote the unique minimum degree vertex in $H$. 
We will follow an approach similar to that used in the proof of Corollary 2.1.4 in~\cite{grinshpun2015some}. 
Before we proceed with the proof, we introduce some notation and facts that we will need later on. 
We begin with the fact that there exists $\gamma_n= o(1)$ such that the following holds:
    \begin{enumerate}
        \item [(1)] For any $d\geq 0$, we have $\Pr(\delta(H) = d) \leq \gamma_n$.
        \item [(2)] For any $d\geq 0$, if $\Hminusu\sim G(n-1,p)$, we have $\Pr(\delta(\Hminusu) \geq  d-1) \geq \Pr(\delta(H) \geq  d) - \gamma_n$.
    \end{enumerate}
    Part (1) follows from the proof of Theorem~3.9(i) in~\cite{bollobas2001random}, while part (2) is shown in the proof of Corollary 2.1.4 in~\cite{grinshpun2015some}. 

Next, let $\eps_n=o(1)$ be chosen such that
$\eps_n=\omega(\max\{\beta_n,\gamma_n,\delta_n\}).$   
We further let $t_n$ be the smallest integer such that $\Pr (\delta(H) \leq t_n) \geq 1-\eps_n$. 
Note that, by the minimality of $t_n$, we then have $\Pr (\delta(H) \leq t_n-1) < 1-\eps_n$. Using (1) for $d = t_n$, we conclude
    \begin{align}\label{eq:mindegreed}
            1-\eps_n \leq \Pr (\delta(H) \leq t_n) = \Pr (\delta(H) \leq t_n-1) + \Pr (\delta(H) = t_n) \leq 1-\eps_n+\gamma_n.
    \end{align}
Moreover, since $\eps_n>\gamma_n + \delta_n$, we obtain
$\Pr (\delta(H) \leq t_n) < 1-\delta_n < \Pr (\delta(H) \leq 2np)$ and thus $t_n\leq 2np$.  
    
    \smallskip
    
    Since $H \sim G(n,p)$, the subgraph $H - v$, for any fixed vertex $v$, has the distribution $G(n-1,p)$. However, recall that we are conditioning on the event $X_{\delta}$, and that in particular there is a unique vertex $u$ of minimum degree $d = d_H(u)$. We will be interested in the subgraph $\Hminusu = H - u$, and first need to determine how conditioning on $X_{\delta}$ affects its distribution.
    
    Suppose $S \subseteq V(\Hminusu)$ is the neighbourhood of $u$. As $u$ is the only vertex of degree at most $d$ in $H$, we must have $d_{\Hminusu}(v) \ge d+1$ for all $v \in V(\Hminusu) \setminus S$, and $d_{\Hminusu}(v) \ge d$ for all $v \in S$; let $C_S$ be the event that these lower bounds on the degrees in $\Hminusu$ hold. Aside from $C_S$, however, $X_{\delta}$ yields no further information about the graph $\Hminusu$, as the edges in $G(n,p)$ are independent. Thus, we have 
    \begin{align}\label{eq:distribution}
        \Pr_{G(n,p)} (H[S] \in \mathcal{P}_d |  X_{\delta} \land \set{N_H(u) = S}) = \Pr_{G(n-1,p)} (\Hminusu[S] \in \mathcal{P}_d |C_S).
    \end{align}

     Now, by the Law of Total Probability,
      \begin{align} \label{eq:sumuptot}
        & \Pr_{G(n,p)}(H[N_H(u)]\in \mathcal{P}_{d_H(u)} |X_{\delta}) \nonumber \\
        &= \sum_{0 \le d \le n-1} \sum_{S \in \binom{V(\Hminusu)}{d}} \Pr_{G(n,p)} \parens*{H[S] \in \mathcal{P}_d | X_{\delta} \land \set{N_H(u) = S} } \cdot \Pr_{G(n,p)} \parens*{N_H(u) = S | X_{\delta} } \nonumber \\
        & \geq \sum_{0.5np \le d \le t_n} \sum_{S \in \binom{V(\Hminusu)}{d}} 
        \Pr_{G(n-1,p)}(\Hminusu[S]\in \mathcal{P}_d| C_S) \cdot \Pr_{G(n,p)}(N_H(u) = S | X_{\delta}).
    \end{align}
     
     To estimate the first factor, we observe that
    \begin{align}\label{eq:C_S}
        \Pr_{G(n-1,p)}(C_S) &\geq \Pr(\delta(\Hminusu)\geq d+1) \geq \Pr(\delta(H)\geq d+2) - \gamma_n \notag \\
        &\geq \Pr(\delta(H)\geq t_n+2) - \gamma_n \geq \Pr(\delta(H)\geq t_n+1) - 2\gamma_n \geq \eps_n/2,
    \end{align}
    where the second inequality follows from~(2), for the third inequality we use $d\leq t_n$, the fourth inequality follows from~(1), and the last inequality comes from \eqref{eq:mindegreed} and since $\eps_n=\omega(\gamma_n)$. Hence we have
    \begin{align*}
        \Pr_{G(n-1,p)}\parens*{ \Hminusu[S] \in \mathcal{P}_d | C_S }
    & = 1 - \Pr_{G(n-1,p)}\left(\Hminusu[S]\notin \mathcal{P}_d|C_S \right) \\
    &= 1 - \frac{\Pr_{G(n-1,p)}\left( \set{\Hminusu[S]\notin \mathcal{P}_d} \wedge C_S\right)}{\Pr_{G(n-1,p)}(C_S)} \\
    & \geq 1 - \frac{\Pr_{G(n-1,p)}\left(\Hminusu[S]\notin \mathcal{P}_d\right)}{\Pr_{G(n-1,p)}(C_S)} \\
    & \geq 1 - \frac{\Pr_{G(d,p)} \left(G_d\notin \mathcal{P}_d\right)}{\eps_n/2} = 1 - o(1),
    \end{align*}
    where for the second inequality we use~\eqref{eq:C_S}
	and that $\Hminusu[S]\sim G(d,p)$ 
	and the final estimate uses~\eqref{beta} and 
	$\beta_n=o(\eps_n)$.    
	Putting this into~\eqref{eq:sumuptot}, we conclude that
  \begin{align*} 
        \Pr_{G(n,p)}(H[N_H(u)]\in \mathcal{P}_{d_H(u)} |X_{\delta})
        & \geq \sum_{0.5np \le d \le t_n} \sum_{S \in \binom{V(\Hminusu)}{d}} (1-o(1)) \Pr_{G(n,p)}(N_H(u) = S | X_{\delta}) \\
        & = (1-o(1)) \Pr_{G(n,p)}\left(0.5np\leq \delta(H) \leq t_n | X_{\delta} \right) \\
        & = (1-o(1)) \frac{\Pr_{G(n,p)}(\set{\delta(H) \leq t_n} \land X_{\delta})}{\Pr(X_\delta)} \\ 
        & \geq (1-o(1)) \frac{1 - \Pr_{G(n,p)}(\delta(H) > t_n) - \Pr_{G(n,p)}(\overline{X_{\delta}})}{\Pr(X_\delta)} \\
        & \geq (1-o(1)) \frac{1-\eps_n-\delta_n}{1-\delta_n} = 1-o(1). 
  \end{align*}
 This proves the lemma.
\end{proof}

\subsection{The smallest neighbourhood and quantitative simplicity}

We can now combine the results from Section~\ref{sec:gnpfacts} with Lemma~\ref{lem:gnp_neighbourhood_general} to obtain a sequence of corollaries describing the subgraph $F$ induced by the neighbourhood of the minimum degree vertex, which we shall later apply when proving Theorem~\ref{thm:main}. We will also use these to derive Corollary~\ref{cor:bounds} from Theorem~\ref{thm:main}.

\medskip

To start with, for the proof of the Ramsey simplicity of $H$ in case~\ref{thm:main:alwayssimple} of Theorem~\ref{thm:main}, it will be important that $F$ is an empty graph. This is guaranteed by the following corollary.

\begin{cor}\label{cor:nhdgraphempty}
Let $p=p(n)\in (0,1)$ be such that $\frac{\log n}{n} \ll p\ll n^{-\frac{2}{3}}$, and let $H\sim G(n,p)$. Then a.a.s.~$H$ has a unique minimum degree vertex $u$, and $e(N_H(u))=0$.
\end{cor}

\begin{proof}
By Lemma~\ref{lem:gnp_neighbourhood_general} it is enough to prove that,
for every $s\in [0.5np,2np]$, with high probability $G_s\sim G(s,p)$ has
no edges. This holds, since by the assumptions on $s$ and $p$ we obtain $\mathbb{E}[e(G_s)]<s^2p\leq 4n^2p^3 = o(1)$.
\end{proof}

For larger values of $p$, we can control the number of edges appearing in $F$, which we will require for the proofs of both simplicity and non-simplicity.

\begin{cor} \label{cor:nhdgraphsparse}
Let $p=p(n)\in (0,1)$ be such that $n^{-\frac{2}{3}} \ll p\ll 1$, and let $H\sim G(n,p)$. Then a.a.s.~$H$ has a unique minimum degree vertex $u$, and the graph $F=H[N(u)]$ satisfies $\frac{1}{16}n^2p^3 \leq e(F)\leq 4n^2p^3$.
\end{cor}

\begin{proof}
For every $s\in [0.5np,2np]$, Lemma~\ref{lem:gnp_edge.counts}\ref{lem:gnp_edge.counts_concentration1} guarantees that $G_s\sim G(s,p)$ almost surely has $(1 + o(1))\frac{s^2p}{2}\in [\frac{1}{16}n^2p^3,4n^2p^3]$ edges. The above statement now follows by an application of Lemma~\ref{lem:gnp_neighbourhood_general}.
\end{proof}

Finally, in the range $n^{-\frac23} \ll p \ll n^{-\frac12}$, when determining the $q$-Ramsey simplicity of $H$, we will make use of the fact that $F$ is typically a forest with small components, while also appealing to the fact that its maximum degree cannot be too small.

\begin{cor} \label{cor:nhdgraphsmalltrees}
Let $p=p(n)\in (0,1)$ be such that $n^{-\frac{2}{3}} \ll p\ll n^{-\frac{1}{2}}$, and let $H\sim G(n,p)$. Then a.a.s.~$H$ has a unique minimum degree vertex $u$, the graph $F=H[N(u)]$ induces a forest, and the order $\lambda(F)$ of the largest component in $F$ satisfies the following bounds:
\begin{enumerate}[leftmargin=40pt,label = (\alph*)]
    \item\label{cor:nhdgraphsmalltrees.regime_log} $\lambda(F) \le \frac12 \log n$,
    \item\label{cor:nhdgraphsmalltrees.regime_constant} if $p\ll n^{-\frac{k+1}{2k+1}}$ for some fixed integer $k \ge 2$, then $\lambda(F) \le k$, and
    \item\label{cor:nhdgraphsmalltrees.regime_intermediate} if $p = n^{-\frac12}f^{-1}$ for some $f = f(n)$ satisfying $1 \ll f = n^{o(1)}$, then $\lambda(F) \le \left(\frac14 + o(1) \right) \frac{\log n}{\log f}$.
\end{enumerate}
Moreover, the maximum degree $\Delta(F)$ of $F$ a.a.s.~satisfies the following:
\begin{enumerate}[leftmargin=40pt, start=4,label = (\alph*)]
    \item\label{cor:nhdgraphsmalltrees.regime_degree1} if $p \gg n^{-\frac{k}{2k-1}}$ for some fixed integer $k \ge 2$, then $\Delta(F) \ge k-1$, and
    \item\label{cor:nhdgraphsmalltrees.regime_degree2} if $p = n^{-1/2}f^{-1}$ for some $1 \ll f = f(n) = n^{o(1)}$, then $\Delta(F) \ge \parens*{\frac12 - o(1)} \frac{\log n}{\log ( f^2 \log n )}$.
\end{enumerate}
\end{cor}

\begin{proof}
By Lemma~\ref{lem:gnp_neighbourhood_general}, it suffices to verify that the corresponding bounds on $\lambda(G_s)$ and $\Delta(G_s)$ for $G_s \sim G(s,p)$ hold with high probability when $s \in [0.5np, 2np]$. These bounds are obtained as follows:
for property~\ref{cor:nhdgraphsmalltrees.regime_log} observe that $p\ll n^{-\frac{1}{2}}$ implies $p\ll s^{-1}$ and $s\ll n^{\frac12}$,
in which case Lemma~\ref{lem:gnp_forest.regime}\ref{lem:gnp_forest.regime_log} gives that $\lambda(G_s)\le \log s \le \frac{1}{2} \log n$ holds a.a.s.. 
For property~\ref{cor:nhdgraphsmalltrees.regime_constant} we use that $p\ll n^{-\frac{k+1}{2k+1}}$ implies $p\ll s^{-\frac{k+1}{k}}$,
and hence $\lambda(G_s)\le k$ holds a.a.s.~by Lemma~\ref{lem:gnp_forest.regime}\ref{lem:gnp_forest.regime_constant}.
For properties~\ref{cor:nhdgraphsmalltrees.regime_intermediate} and~\ref{cor:nhdgraphsmalltrees.regime_degree2}
observe that $p = n^{-\frac12}f^{-1}$ implies $\frac{0.5}{sf^2} \le p\le \frac{2}{sf^2}$
and $s=n^{\frac12 -o(1)}$,
which a.a.s.~leads to $\lambda(G_s)\le (1+o(1))\frac{\log s}{\log f^2} \le  \left(\frac14 + o(1) \right) \frac{\log n}{\log f}$ by Lemma~\ref{lem:gnp_forest.regime}\ref{lem:gnp_forest.regime_intermediate},
and to
$\Delta(G_s) \geq (1-o(1))\frac{\log s}{\log (f^2 \log s)} \ge \parens*{\frac12 - o(1)} \frac{\log n}{\log ( f^2 \log n )}$
by Lemma~\ref{lem:gnp_degrees}\ref{lem:gnp_degree_max2}.
Finally, for property~\ref{cor:nhdgraphsmalltrees.regime_degree1} we note that $p\gg n^{- \frac{k}{2k-1}}$ implies $p\gg s^{-\frac{k}{k-1}}$, and hence Lemma~\ref{lem:gnp_degrees}\ref{lem:gnp_degree_max1} ensures that $\Delta(G_s)\geq k-1$ a.a.s.
\end{proof}

With these bounds on the parameters of the subgraph $F$ induced by the neighbourhood of the minimum degree vertex, we are now in position to deduce Corollary~\ref{cor:bounds}, giving quantitative estimates on the value of $\tilde{q}(H)$ in the intermediate range.

\begin{proof}[Proof of Corollary~\ref{cor:bounds}]
We start by appealing to Lemma~\ref{lem:gnp_degrees}\ref{lem:gnp_degree_concentration} to observe that a.a.s.~$\delta(H) = (1 + o(1))np$.

Let us begin by establishing the lower bounds on $\tilde{q}(H)$. By Theorem~\ref{thm:main} we have $\tilde{q}(H) \ge (1 + o(1)) \max \set*{ \frac{\delta(H)}{\lambda(F)^2}, \frac{\delta(H)}{80 \log n}}$, and we can bound $\lambda(F)$ using Corollary~\ref{cor:nhdgraphsmalltrees}.

When $p \ll n^{-\frac{k+1}{2k+1}}$ for some fixed integer $k \ge 2$, then, by Corollary~\ref{cor:nhdgraphsmalltrees}\ref{cor:nhdgraphsmalltrees.regime_constant}, we a.a.s.~have $\lambda(F) \le k$. Thus, in this range, we have $\tilde{q}(H) \ge (1 + o(1)) \frac{np}{k^2}$ a.a.s., which yields the lower bounds for parts~\ref{cor:bounds:constanttreecomponents} and~\ref{cor:bounds:boundary} of Corollary~\ref{cor:bounds} (note that when $p = \Theta \left( n^{-\frac{k+1}{2k+1}} \right)$, we have $p \ll n^{-\frac{(k+1)+1}{2(k+1)+1}}$). The lower bound in part~\ref{cor:bounds:smalldegrees} follows by substituting the bound on $\lambda(F)$ from Corollary~\ref{cor:nhdgraphsmalltrees}\ref{cor:nhdgraphsmalltrees.regime_intermediate}, while the lower bound in part~\ref{cor:bounds:manyedges} is trivial.

For the upper bounds, Theorem~\ref{thm:main} gives $\tilde{q}(H) \le \min \set*{\frac{\delta(H)}{\Delta(F)}, \frac{\delta(H)^2}{2e(F)}}$. The upper bounds in parts~\ref{cor:bounds:constanttreecomponents},~\ref{cor:bounds:boundary}, and~\ref{cor:bounds:smalldegrees} come from substituting the appropriate lower bounds on $\Delta(F)$ given by Corollary~\ref{cor:nhdgraphsmalltrees}. When $p \gg n^{-\frac{k}{2k-1}}$ for some fixed $k$, Corollary~\ref{cor:nhdgraphsmalltrees}\ref{cor:nhdgraphsmalltrees.regime_degree1} yields $\Delta(F) \ge k-1$ a.a.s., which provides the upper bounds in parts~\ref{cor:bounds:constanttreecomponents} and~\ref{cor:bounds:boundary} of Corollary~\ref{cor:bounds}. The upper bound in part~\ref{cor:bounds:smalldegrees} follows similarly, using the lower bound on $\Delta(F)$ from Corollary~\ref{cor:nhdgraphsmalltrees}\ref{cor:nhdgraphsmalltrees.regime_degree2}. Finally, for the upper bound in part~\ref{cor:bounds:manyedges} of Corollary~\ref{cor:bounds}, we use Corollary~\ref{cor:nhdgraphsparse}, which asserts that a.a.s.~$e(F) \ge \frac{1}{16}n^2 p^3$. Thus $\frac{\delta(H)^2}{2e(F)} \le \frac{8 + o(1)}{p}$, as required.
\end{proof}

\section{Simplicity and abundance for $G(n,p)$} \label{sec:simplicity}

In this section we prove the lower bounds on $\tilde{q}(G(n,p))$ from Theorem~\ref{thm:main}. These are the positive results, showing that with high probability $H \sim G(n,p)$ is $q$-Ramsey simple for the appropriate values of $q$.

To begin, we observe that we have nothing new to prove in case~\ref{thm:main:forest}. By Lemma~\ref{lem:gnp_forest.regime} we know $H$ is a forest with high probability when $p \ll n^{-1}$. Szab\'o, Zumstein, and Z\"urcher~\cite{szabo2010minimum} proved that all forests are $2$-Ramsey simple, and their proof extends directly to show $q$-Ramsey simplicity for all $q \ge 3$ as well. For completeness, we provide the argument in Appendix~\ref{sec:forests}.

For the remaining cases, we will show that $H$ is typically such that one can construct a minimal $q$-Ramsey graph $G$ for $H$ with $\delta(G) = q(\delta(H) - 1) + 1$, provided, in case~\ref{thm:main:intermediatelowerbounds}, that $q$ is not too large. We first establish a general sufficient condition for the existence of such a graph $G$ in Section~\ref{sec:simpreduction}, and then show in Section~\ref{sec:simpconstruction} that it is satisfied with high probability by the random graph $H$. In Section~\ref{sec:simabundance} we shall extend these results by showing $H$ admits minimal Ramsey graphs with arbitrarily many vertices of degree $q (\delta(H) - 1) + 1$.

Before we start, we introduce a piece of notation we shall use throughout this section. Given a graph $\Gamma$ with a $q$-colouring $f: E(\Gamma) \rightarrow [q]$ and any colour $i\in [q]$, the colour-$i$ subgraph $\Gamma_i$ of $\Gamma$ is the graph $\Gamma_i = (V(\Gamma), f^{-1}(i))$ consisting of all edges of $\Gamma$ with the colour $i$.

\subsection{Reducing to the smallest neighbourhood} \label{sec:simpreduction}

In this subsection we shall show that when establishing the $q$-Ramsey simplicity of a well-behaved graph $H$ (recall Definition~\ref{def:wellbehaved}), we can focus our attention on the neighbourhood of the minimum degree vertex.

\begin{prop} \label{prop:nhdgraph}
Let $q \ge 2$, let $H$ be a well-behaved graph, and let $F = H[N(u)]$ be the subgraph induced by the neighbourhood of the unique minimum degree vertex $u$. Suppose there exists a $q$-edge-coloured graph $\Gamma$ on $q(\delta(H) - 1) + 1$ vertices such that:
\begin{enumerate}[label = (\roman*)]
	\item \label{prop:nhdgraph:Fcopies} for every set $U \subseteq V(\Gamma)$ of $\delta(H)$ vertices and for every colour $i \in [q]$, there exists a copy $F_{U,i}$ of $F$ in $\Gamma[U]$ whose edges are all of colour $i$, and
	\item \label{prop:nhdgraph:maxdeg} for each $i \in [q]$, the colour-$i$ subgraph $\Gamma_i$ of $\Gamma$ has maximum degree at most $\delta(H) - 1$.
\end{enumerate}
Then $H$ is $q$-Ramsey simple.
\end{prop}

This proposition provides a sufficient condition: to establish the $q$-Ramsey simplicity of a well-behaved graph, one need only construct the coloured graph $\Gamma$. Before proceeding with its proof, we remark that the condition is very close to being necessary as well.

\begin{rem} \label{rem:nhdgraphnecessary}
Let $H$ be $q$-Ramsey simple with a unique vertex $u$ of minimum degree, and let $G$ be a minimal $q$-Ramsey graph for $H$ with a vertex $w$ of degree $q(\delta(H) - 1) + 1$. Let $\Gamma = G[N(w)]$ be the subgraph of $G$ induced by the neighbourhood of $w$. By minimality, there is a $q$-colouring $c$ of $G - w$, and in particular of $\Gamma$, without any monochromatic copies of $H$. 

Since $G$ itself is $q$-Ramsey for $H$, no matter how we extend the colouring $c$ to the edges incident to $w$, we must create a monochromatic copy of $H$. Given any subset $U$ of $\delta(H)$ vertices in $\Gamma$ and any colour $i \in [q]$, colour the edges from $w$ to $U$ with colour $i$, and colour the remaining edges incident to $w$ evenly with the other colours, so that each is used $\delta(H) - 1$ times. Any monochromatic copy of $H$ must involve at least $\delta(H)$ edges incident to $w$, and hence must be of colour $i$ and contain all the vertices in $U$. As $w$ has degree $\delta(H)$ in this monochromatic subgraph, it must play the role of $u$ in $H$, and therefore we must find a colour-$i$ copy of $F$ in $\Gamma[U]$.

Thus, if $H$ is $q$-Ramsey simple, there must exist a $q$-coloured graph $\Gamma$ on $q(\delta(H) - 1)+1$ vertices satisfying property~\ref{prop:nhdgraph:Fcopies} of Proposition~\ref{prop:nhdgraph}. While the well-behavedness of $H$ and property~\ref{prop:nhdgraph:maxdeg} may not be necessary, they shall enable us to maintain control over potential copies of $H$ when constructing the minimal $q$-Ramsey graph $G$.
\end{rem}

Given the graph $\Gamma$, when we build from it a $q$-Ramsey graph $G$ we shall, as is common practice in the field, make extensive use of signal senders, which are gadgets that allow us to prescribe colour patterns on the edges of a graph.

\begin{definition}[Signal senders] \label{def:signalsenders}
Given a graph $H$, a number of colours $q \ge 2$, a distance $d \ge 1$, and two edges $e$ and $f$, a \emph{positive (or negative) signal sender} $S^+(H,q,d,e,f)$ (or $S^-(H,q,d,e,f)$) is a graph $S$ that contains $e$ and $f$ and satisfies:
\begin{enumerate}[label = (\roman*)]
	\item \label{def:signalsenders:Hfree} $S$ can be $q$-coloured without monochromatic copies of $H$,
	\item \label{def:signalsenders:signaledges} in any such colouring, $e$ and $f$ have the same (or different) colour(s), and
	\item \label{def:signalsenders:distance} the edges $e$ and $f$ are at distance at least $d$ in $S$.
\end{enumerate}
The edges $e$ and $f$ are called the \emph{signal edges}.
\end{definition}

Fortunately for us, signal senders exist for all $3$-connected graphs, as shown by R\"odl and Siggers~\cite{rodl2008ramseyminimal}, building on earlier work of Burr, Erd\H{o}s, and Lov\'asz~\cite{burr1976graphs} and Burr, Ne\v{s}et\v{r}il, and R\"{o}dl~\cite{burr1985useofsenders}.

\begin{theorem}[\cite{rodl2008ramseyminimal}] \label{thm:signalsenders}
If $H$ is $3$-connected, then for any $q \ge 2$ and $d \ge 1$, there are positive and negative signal senders $S^+(H,q,d,e,f)$ and $S^-(H,q,d,e,f)$.
\end{theorem}

The utility of signal senders lies in the ability to force pairs of edges in an $H$-free colouring of a graph $G$ to have the same (or different, in the negative case) colours. This is achieved through the process of \emph{attachment}; given a graph $G$ and a pair of distinct edges $h_1, h_2 \in E(G)$, we attach to $G$ a signal sender $S^+(H,q,d,e,f)$ (or $S^-(H,q,d,e,f)$), defined on a disjoint set of vertices, between $h_1$ and $h_2$ by identifying the signal edges $e$ and $f \in E(S)$ with the edges $h_1$ and $h_2 \in E(G)$. In this next result, we show that attachment cannot create unexpected copies of our target graph $H$, provided that the signal edges are sufficiently far apart.

\begin{lemma}\label{lem:safety}
Let $q\geq 2$, let $H$ be any 3-connected graph, and let $d \ge v(H)$. Let $S = S^+(H,q,d,e,f)$ or $S = S^-(H,q,d,e,f)$ be a signal sender and let $G$ be any graph on a disjoint set of vertices. If the graph $G'$ is formed by attaching $S$ to any two distinct edges of $G$, then, for any copy $H_0$ of $H$ in $G'$, we have either $V(H_0)\subseteq V(G)$ or $V(H_0)\subseteq V(S)$.
\end{lemma}

\begin{proof}
Let $H_0$ be a copy of $H$ in $G$ and suppose for the sake of contradiction that $H_0$ is fully contained neither in $G$ nor in $S$. We can then find vertices $x \in V(H_0) \cap \left( V(S) \setminus V(G) \right)$ and $y \in V(H_0) \cap \left( V(G) \setminus V(S) \right)$. Now, by $3$-connectivity, $H_0$ contains three internally-vertex-disjoint paths between $x$ and $y$.

Since $V(S) \cap V(G) = e \cup f$, each of these paths must pass through a distinct endpoint of one of the signal edges $e$ and $f$. There must be one path meeting $e$ and another meeting $f$, and the portions of these paths that lie within the signal sender contain a path from $e$ to $f$ within $V(H_0) \cap V(S)$. However, this contradicts $e$ and $f$ being at distance $d \ge v(H)$.
\end{proof}

Armed with these preliminaries, we can now prove Proposition~\ref{prop:nhdgraph}.

\begin{proof}[Proof of Proposition~\ref{prop:nhdgraph}]
We shall take a slightly indirect route to certifying the $q$-Ramsey simplicity of $H$. Rather than constructing a minimal $q$-Ramsey  graph with minimum degree $q(\delta(H) - 1)+1$, we will instead build a graph $G$ such that:
\begin{itemize}
	\item[(a)] $G \rightarrow_q H$,
	\item[(b)] $G$ has a vertex $w$ of degree $q(\delta(H) - 1) + 1$, and
	\item[(c)] $G - w \not\rightarrow_q H$.
\end{itemize}
Since $G$ is $q$-Ramsey for $H$, it must contain a minimal $q$-Ramsey subgraph $G' \subseteq G$. By virtue of (c), we have $w \in V(G')$, and hence $\delta(G') \le d_{G'}(w) \le d_G(w) = q(\delta(H) - 1)+1$. In light of the general lower bound, we must in fact have equality, and hence $G'$ bears witness to the $q$-Ramsey simplicity of $H$.

\smallskip

To construct this $q$-Ramsey graph $G$, we start with the graph $\Gamma$. Recall that, for each set $U$ of $\delta(H)$ vertices of $\Gamma$ and for each colour $i \in [q]$, there is a colour-$i$ copy $F_{U,i}$ of $F$ in $\Gamma[U]$. We will wish to complete these to potential monochromatic copies of $H$. To this end, let $R = H - \left(\set{u} \cup N(u) \right)$ be the remainder of $H$ after we remove the minimum degree vertex $u$ and its neighbourhood. Then, for every $U$ and $i$, we include a copy $R_{U,i}$ of $R$ on a disjoint set of vertices, adding the necessary edges so that $R_{U,i} \cup F_{U,i}$ forms a copy of $H - u$. We call the resulting graph $\Gamma^+$.

Now recall that the graph $\Gamma$ comes with an edge-colouring, which we extend by colouring the edges in $R_{U,i}$ and between $R_{U,i}$ and $F_{U,i}$ with the colour $i$. Denote by $c$ the resulting colouring of $\Gamma^+$. To force the correct colouring, we shall use signal senders. Note that, since $H$ is well-behaved, property~\ref{def:wb_3conn} ensures $H$ is $3$-connected, and hence by Theorem~\ref{thm:signalsenders} positive and negative signal senders exist. 

We introduce a matching $e_1, e_2, \hdots, e_q$ of $q$ edges, again on a set of new vertices. For every pair $i < j$, we attach a negative signal sender $S_{i,j} = S^-(H,q,v(H),e_i,e_j)$ between $e_i$ and $e_j$. As we shall see later, this will ensure that these edges all receive distinct colours in an $H$-free colouring. Now, for every edge $f$ in $\Gamma^+$, we attach a positive signal sender $S_f = S^+(H,q,v(H),e_{c(f)},f)$ between $e_{c(f)}$ and $f$. Finally, we introduce a new vertex $w$ and make it adjacent to every vertex in $\Gamma$. This completes our construction of the graph $G$, which is depicted in Figure~\ref{fig:construction-G}.

\begin{figure}[h!]
    \centering
    \includegraphics[scale=0.6,page=1]{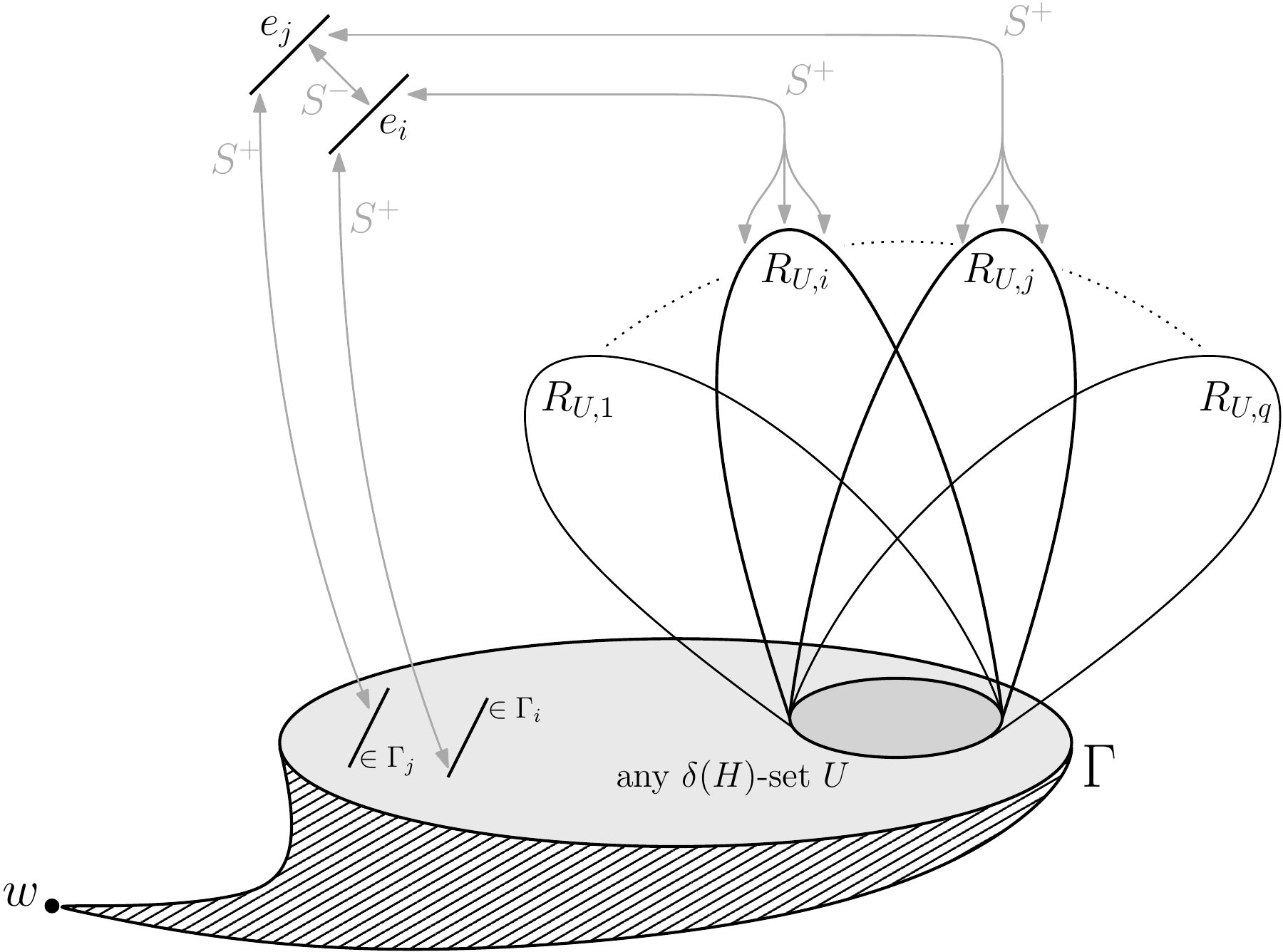}
    \caption{Construction of $G$}
    \label{fig:construction-G}
\end{figure}

Observe that $d_G(w) = v(\Gamma) = q(\delta(H) - 1) + 1$, and so condition (b) is already satisfied. We shall now verify conditions (a) and (c) in the following claims.

\begin{clm} \label{clm:nhdgraphisramsey}
The graph $G$ is $q$-Ramsey for $H$.
\end{clm}

\begin{proof}
Suppose for a contradiction that we have an $H$-free $q$-colouring of $G$. First observe that, by Definition~\ref{def:signalsenders}\ref{def:signalsenders:signaledges}, if the signal sender $S_{i,j}$ is $H$-free, then the edges $e_i$ and $e_j$ must receive different colours. As this is true for each pair $i < j$, we may, relabelling colours if necessary, assume that each edge $e_i$ receives colour $i$.

Next, for each edge $f$ in $\Gamma^+$, consider the signal sender $S_f$. If this does not contain a monochromatic copy of $H$, then $e_{c(f)}$ and $f$ must have the same colour, and thus $f$ receives the colour $c(f)$. Hence we have forced the desired colouring on $\Gamma^+$.

This brings us to the vertex $w$. Since it has degree $q(\delta(H) - 1) + 1$, there must be some colour $i$ and a set $U \subseteq V(\Gamma)$ of size $\delta(H)$ such that the edges between $w$ and $U$ are all of colour $i$. However, appealing to condition~\ref{prop:nhdgraph:Fcopies} of Proposition~\ref{prop:nhdgraph}, we find a colour-$i$ copy $F_{U,i}$ of $F$ in $\Gamma[U]$, which we can complete to a copy of $H$ by attaching $w$ and $R_{U,i}$, contradicting our supposition. 
\end{proof}

\begin{clm} \label{clm:nhdgraphisnotramsey}
The graph $G - w$ is not $q$-Ramsey for $H$.
\end{clm}

\begin{proof}
We provide an $H$-free $q$-colouring of $G - w$. To start, we give $\Gamma^+$ the colouring $c$, and, for each $i \in [q]$, colour the edge $e_i$ of the matching with the colour $i$. Observe that, under this colouring, the signal edges of each positive signal sender $S_f$ in $G$ have the same colour, while those of negative signal senders $S_{i,j}$ receive different colours. By Definition~\ref{def:signalsenders} we can find an $H$-free colouring of each signal sender that agrees with the colouring of the signal edges. We use these to extend our colouring to the signal senders as well, thereby obtaining a $q$-colouring of $G - w$.

Now suppose for a contradiction that this colouring gives rise to a colour-$i$ copy $H_0$ of $H$ for some $i \in [q]$. First, observe that it follows from Lemma~\ref{lem:safety} that $H_0$ either is fully contained in a signal sender or is contained in $\Gamma^+ \cup \{ e_i : i \in [q] \}$. Since the signal senders were coloured without monochromatic copies of $H$, and the edges $\{e_i: i \in [q]\}$ are isolated in the latter graph, we need only show that we cannot have $H_0 \subseteq \Gamma^+$.

We next claim that $H_0$ can only meet at most one subgraph $R_{U,i}$. Indeed, suppose instead that there are two sets $U$ and $U'$ such that $V(H_0) \cap V(R_{U,i})$ and $V(H_0) \cap V(R_{U',i})$ are both nonempty. As the sets $V(R_{U,i})$ and $V(R_{U',i})$ are disjoint, we may assume without loss of generality that $\card{V(H_0) \cap V(R_{U,i})} \le \tfrac12 n$. 

Since $R_{U,i}$ is only attached to $\Gamma$ through the vertices in $U$, the set $U$ must be a cut-set for the subgraph $H_0$. Let $x \in V(H_0) \cap V(R_{U,i})$ be an arbitrary vertex, and let $K$ be the component of $x$ in $H_0 - U$. We clearly have $\card{K} \le \card{V(H_0) \cap V(R_{U,i})} \le \tfrac12 n$.

On the other hand, observe that $x$ is also in the copy $H_{U,i}$ of $H$ supported on $\{w\} \cup V(F_{U,i}) \cup V(R_{U,i})$. In $H_{U,i}$, the set $U$ is the neighbourhood of $w$, and, since $H$ is well-behaved, condition~\ref{def:wb_pair_degree} implies $x$ has at most $\tfrac12 \delta(H)$ neighbours in $U$. As $d_{H_0}(x) \ge \delta(H)$, this means $x$ must have at least $\tfrac12 \delta(H)$ neighbours in $H_0 - U$. Hence, we also have $\card{K} \ge \tfrac12 \delta(H)$. However, this contradicts condition~\ref{def:wb_components}, as the removal of the $\delta(H)$ vertices in $U$ cannot create a component in $H_0$ of size between $\tfrac12 \delta(H)$ and $\tfrac12 n$.

Thus, $H_0$ meets at most one subgraph $R_{U,i}$. Now, by property (ii) of the colouring $c$ of $\Gamma$, we have that any vertex is incident to fewer than $\delta(H)$ edges of colour $i$ in $\Gamma$. Thus, in order to be part of $H_0$, a vertex from $\Gamma$ must have neighbours in $R_{U,i}$ as well. However, the only such vertices are those in $U$, and since $\card{U \cup V(R_{U,i})} = n-1$, this does not leave us with enough vertices for a copy of $H$.

Our colouring is therefore indeed $H$-free, thereby proving the claim.
\end{proof}

This shows that the graph $G$ satisfies conditions (a), (b), and (c), completing the proof.
\end{proof}

\subsection{Constructing coloured neighbourhoods} \label{sec:simpconstruction}

The path to proving the lower bounds of Theorem~\ref{thm:main} is now clearly signposted. By Lemma~\ref{lem:wellbehaved}, we know that when $\frac{\log n}{n} \ll p \ll 1$, the random graph $H \sim G(n,p)$ is well-behaved with high probability, and hence we are in position to apply Proposition~\ref{prop:nhdgraph}. We shall then use the results of Section~\ref{sec:gnpproperties} to describe the subgraph $F$ induced by the minimum degree vertex in $H$. This subgraph evolves as the edge probability $p$ increases, and in each range we will construct an appropriate coloured graph $\Gamma$ that satisfies the conditions of the proposition.

We start with the sparse range, where $p \ll n^{-\frac23}$.

\begin{proof}[Proof of Theorem~\ref{thm:main}\ref{thm:main:alwayssimple}]
Let $q \ge 2$, let $p$ satisfy $\frac{\log n}{n} \ll p \ll n^{-\frac23}$, and let $H \sim G(n,p)$. By Lemma~\ref{lem:wellbehaved} and Corollary~\ref{cor:nhdgraphempty}, we have with high probability that $H$ is well-behaved and the subgraph $F = H[N(u)]$ induced by the neighbourhood of the minimum degree vertex $u$ is empty. In this case, we can simply take $\Gamma$ to be an empty graph on $q(\delta(H) - 1) + 1$ vertices. Properties~\ref{prop:nhdgraph:Fcopies} and~\ref{prop:nhdgraph:maxdeg} of Proposition~\ref{prop:nhdgraph} are then trivially satisfied, and so it follows that $H$ is $q$-Ramsey simple.
\end{proof}

When $p \gg n^{-\frac23}$, we will begin to see edges in the neighbourhood of the minimum degree vertex. Provided $p \ll n^{-\frac12}$, though, the neighbourhood remains simple in structure, and we can get reasonably sharp bounds on the number of colours for which the random graph is Ramsey simple.

\begin{proof}[Proof of Theorem~\ref{thm:main}\ref{thm:main:intermediatelowerbounds}, first bound]
Let $n^{-2/3}\ll p\ll n^{-1/2}$ and $H \sim G(n,p)$. By Lemma~\ref{lem:wellbehaved}, we know that with high probability $H$ is well-behaved. Let $\lambda(F)$ be the order of the largest component of the subgraph $F = H[N_H(u)]$ induced by the neighbourhood of the minimum degree vertex $u$. Given any $\eps > 0$, we shall show that, as $n$ tends to infinity, $H$ is with high probability $q$-Ramsey simple for every $q \le \parens*{1 - 5 \eps} \frac{\delta(H)}{\lambda(F)^2}$.

By Corollaries~\ref{cor:nhdgraphsparse} and~\ref{cor:nhdgraphsmalltrees} the graph $F$ is with high probability a very sparse forest. More precisely, if we denote by $T_1, T_2, \hdots, T_t$ the components of $F$ that contain at least one edge, then each $T_j$ is a tree spanning at most $\lambda(F)$ vertices and $\sum_j v(T_j) \le \eps \delta(H)$.

To prove simplicity, we provide a geometric construction of an edge-coloured graph $\Gamma$ on $q(\delta(H) - 1)+1$ vertices. Let $s$ be the largest prime number that is at most $\parens*{1 - \eps} \frac{\delta(H)}{\lambda(F)}$. By the upper bound of Baker, Harman, and Pintz~\cite{baker2001primegaps} on prime gaps, we have $s \ge \parens*{1 - 2 \eps} \frac{\delta(H)}{\lambda(F)}$. Now consider the finite affine plane $\mathbb{F}_s^2$, which has $s^2$ points. Each line in the plane consists of $s$ points, and the set of lines can be partitioned into $s+1$ parallel classes $C_1, C_2, \hdots, C_{s+1}$ of $s$ lines each.

To form the graph $\Gamma$, we take as vertices an arbitrary set of $q(\delta(H) - 1) + 1$ points from $\mathbb{F}_s^2$. Note that our choices of $q$ and $s$ ensure that $q(\delta(H) - 1) + 1 \le s^2$ and $q \leq s \le \delta(H)$. Then, given $x, y \in V(\Gamma)$, we add the edge $\{x,y\}$ if and only if the line they span lies in one of the first $q$ parallel classes. We colour the edges by the parallel classes; that is, if the corresponding line lies in $C_i$, for some $i \in [q]$, we give the edge $\{x,y\}$ the colour $i$.

We shall now show that $\Gamma$ satisfies properties~\ref{prop:nhdgraph:Fcopies} and~\ref{prop:nhdgraph:maxdeg} of Proposition~\ref{prop:nhdgraph}, which will show that $H$ is $q$-Ramsey simple. We start with the latter property. The colour-$i$ subgraph $\Gamma_i$ of $\Gamma$ consists of pairs of points in lines in the parallel class $C_i$. Each such line gives rise to a clique in $\Gamma$, and since the lines are parallel, these cliques are vertex-disjoint. Finally, since each line has at most $s$ points in $\Gamma$, it follows that $\Delta(\Gamma_i) \le s-1 \le \delta(H) - 1$, and hence property~\ref{prop:nhdgraph:maxdeg} holds.

For property~\ref{prop:nhdgraph:Fcopies}, we need to show that for any $\delta(H)$-set $U \subseteq V(\Gamma)$ and any colour $i \in [q]$, we can find a copy of $F$ in $\Gamma_i[U]$. We shall embed the trees $T_j$ one at a time. Suppose, for some $j \ge 1$, we have already embedded $T_1, T_2, \hdots, T_{j-1}$, and let $U' \subseteq U$ be the set of vertices we have not yet used. Since $F$ has at most $\eps \delta(H)$ non-isolated vertices, it follows that $\card{U'} \ge \parens*{1 - \eps} \delta(H)$.

As observed when showing property~\ref{prop:nhdgraph:maxdeg}, the colour-$i$ subgraph $\Gamma_i$ is a disjoint union of at most $s$ cliques. Hence, by the pigeonhole principle, $U'$ meets one of these cliques in at least $\frac{\card{U'}}{s}$ vertices. By our choice of $s$, this is at least $\lambda(F)$, and so $\Gamma_i[U']$ contains a clique on $\lambda(F)$ vertices, in which we can freely embed $T_j$.

Repeating this process, we can embed all the trees, thereby obtaining a copy of $F$ in $\Gamma_i[U]$. Hence property~\ref{prop:nhdgraph:Fcopies} is satisfied as well, and thus $H$ is indeed $q$-Ramsey simple.
\end{proof}

The above construction allows us to obtain lower bounds on $\tilde{q}(H)$ whenever $n^{-2/3}\ll p\ll n^{-1/2}$. However, when $p = n^{-\frac12 - o(1)}$ and $\lambda(F)$ gets larger, a probabilistic construction yields a better bound. 

\begin{proof}[Proof of Theorem~\ref{thm:main}\ref{thm:main:intermediatelowerbounds}, second bound]
Let $p \ll n^{-\frac12}$, and let $H \sim G(n,p)$. Our goal is to show that if $q \le \frac{\delta(H)}{80 \log n}$, then with high probability $H$ is $q$-Ramsey simple. We again start by collecting some information about the random graph $H$, before constructing an appropriate graph $\Gamma$ for Proposition~\ref{prop:nhdgraph}.

By Lemma~\ref{lem:gnp_degrees}\ref{lem:gnp_degree_concentration} and Lemma~\ref{lem:wellbehaved}, we may assume that $H$ is well-behaved with $\delta(H) = (1\pm o(1)) np$.
Furthermore, applying 
Corollaries~\ref{cor:nhdgraphsparse} and~\ref{cor:nhdgraphsmalltrees}, we know that with high probability, the subgraph $F = H[N(u)]$ induced by the neighbourhood of the minimum degree vertex is a forest with $o\parens*{\delta(H)}$ edges containing no tree on more than $\log n$ vertices. We label the components of $F$ as $T_1, T_2, \hdots, T_t$.

We now define the $q$-coloured graph $\Gamma$ on $N = q(\delta(H) - 1)+1$ vertices. We take $\Gamma \sim G(N,\tfrac{1}{2})$ to be a random graph with edge probability $\tfrac{1}{2}$. Once we have sampled the graph, we also equip it with a random colouring, colouring each edge independently and uniformly at random from the $q$ colours.

Observe that for each colour $i \in [q]$, the colour-$i$ subgraph $\Gamma_i \subseteq \Gamma$ has the distribution $G(N,\tfrac{1}{2q})$. Hence, it follows from Lemma~\ref{lem:gnp_degrees}\ref{lem:gnp_degree_concentration}, combined with a union bound over the number of colours $q$, that with high probability $\Delta(\Gamma_i) \le \parens*{1 + o(1)} \tfrac{N}{2q} < \delta(H)$ for every $i\in [q]$. This establishes property~\ref{prop:nhdgraph:maxdeg} of Proposition~\ref{prop:nhdgraph}.

We now need to show that property~\ref{prop:nhdgraph:Fcopies} also holds with high probability. That is, we need to ensure that, for every colour $i \in [q]$ and every set $U \subseteq V(\Gamma)$ of $\delta(H)$ vertices, we can find a copy of $F$ in $\Gamma_i[U]$. We shall once again do this by proving the stronger fact that, taking $\varepsilon \ge 0$, for any set $U'$ of $(1- \varepsilon) \delta(H) \ge \tfrac12 np$ vertices, and any tree $T$ on at most $\log n$ vertices, we can embed a copy of $T$ in $\Gamma_i[U']$. We can then greedily embed the components of $F$ one at a time; as $F$ only has $o(\delta(H))$ edges, we will always have at least $(1- \varepsilon) \delta(H)$ vertices remaining when embedding one of its components.

Applying Lemma~\ref{lem:gnp_edge.counts}\ref{lem:gnp_edge.counts_concentration2} combined with a union bound over the colours $i \in [q]$,
we know that with high probability the monochromatic subgraphs $\Gamma_i$ have the property that the number of edges spanned by any set of $\tfrac12 np$ vertices is at least $\tfrac14 \parens*{\tfrac12 np}^2 \tfrac{1}{2q} > 2 np \log n$.

Since the set $U'$ spans at least $2 np \log n$ edges, the average degree in any such subgraph is at least $2 \log n$. By repeatedly removing low-degree vertices, we obtain a subgraph with minimum degree at least $\log n$. It is then trivial to embed a tree on at most $\log n$ vertices in this subgraph, as at each vertex, we will always have enough unused neighbours to embed its children. Thus, we can find disjoint copies of the trees $T_1, T_2, \hdots, T_t$, thereby constructing a copy of $F$ in $\Gamma_i[U]$. This proves property~\ref{prop:nhdgraph:Fcopies}, and so by Proposition~\ref{prop:nhdgraph} it follows that $H$ is $q$-Ramsey simple.
\end{proof}

\subsection{Abundance} \label{sec:simabundance}

Proposition~\ref{prop:nhdgraph} shows that, when establishing the $q$-Ramsey simplicity of a graph $H$, it suffices to consider the neighbourhood of a minimum degree vertex $w$. In the construction of the Ramsey host graph $G$, the vertex $w$ will have the desired degree $\delta(G) = s_q(H)$, but we can expect all other vertices to have much higher degree. Indeed, they are all contained in signal senders, which tend to be large and complicated structures. It is then natural to ask if this must be the case, or if we can instead find minimal $q$-Ramsey graphs for $H$ with arbitrarily many vertices of the lowest possible degree. Following the terminology of~\cite{BCG2020a}, we say a graph $H$ is \emph{$s_q$-abundant} if, for every $k \ge 1$, there is a minimal $q$-Ramsey graph for $H$ with at least $k$ vertices of degree $s_q(H)$. In this section we shall extend the results of the previous section, showing that in cases~\ref{thm:main:alwayssimple} and~\ref{thm:main:intermediatelowerbounds}, $G(n,p)$ is almost surely not just $q$-Ramsey simple but also $s_q$-abundant.

\begin{prop}\label{prop:abundance}
Let $q \ge 2$ and let $H$ be a well-behaved $n$-vertex graph. If there is a $q$-edge-coloured graph $\Gamma$ on $q(\delta(H) - 1) + 1$ vertices satisfying the conditions of Proposition~\ref{prop:nhdgraph}, and if either $e(\Gamma) = 0$ or $n > q(\delta(H) - 1) + 2$, then not only is $H$ $q$-Ramsey simple, but it is also $s_q$-abundant.
\end{prop}

As we have shown in the previous section, for the ranges of parameters covered by cases~\ref{thm:main:alwayssimple} and~\ref{thm:main:intermediatelowerbounds} of Theorem~\ref{thm:main}, $H$ is well-behaved and admits the construction of a suitable $q$-coloured graph $\Gamma$. Moreover, when $p \ll n^{-2/3}$, we have $e(\Gamma) = 0$, while when $n^{-2/3} \ll p \ll n^{-1/2}$, we have $\delta(H) = (1 + o(1)) np$ and $q \le np$, and so $q(\delta(H) - 1) + 2 \le (1 + o(1))(np)^2 \ll n$. Hence, once we prove Proposition~\ref{prop:abundance}, we will have shown that in these cases $G(n,p)$ is also $s_{q}$-abundant. To do so, we shall apply the following theorem, a simple corollary of Theorem 3.1 from~\cite{BCG2020a}, which gives a sufficient condition for the existence of minimal $q$-Ramsey graphs with several vertices of a given degree.

\begin{theorem} \label{thm:arbitrarily_many_3-connected}
Let $H$ be $3$-connected and assume there exists a minimal $q$-Ramsey graph $G'$ for $H$, together with a vertex $v_0 \in V(G')$ and an edge $e\in E(G')$ such that $v_0$ end $e$ do not share a copy of $H$ in $G'$. Then, for any $k \ge 1$, there exists a minimal $q$-Ramsey graph for $H$ that has at least $k$ vertices of degree $d_{G'}(v_0)$.
\end{theorem}

To prove Proposition~\ref{prop:abundance}, we shall show that the $q$-Ramsey graph $G$ we built in the proof of Proposition~\ref{prop:nhdgraph} admits a subgraph $G' \subseteq G$ satisfying conditions of Theorem~\ref{thm:arbitrarily_many_3-connected} when we take $v_0$ to be the minimum degree vertex $w \in V(G)$, implying that $H$ is $s_q$-abundant.

\begin{proof}[Proof of Proposition~\ref{prop:abundance}]
Consider the graph $G$ constructed in the proof of Proposition~\ref{prop:nhdgraph}, and recall that it in particular contained a vertex $w$ of degree $q(\delta(H) - 1) + 1$, and a matching $M = \{e_1, e_2,\hdots, e_q\}$ of edges that were attached to the rest of the graph by signal senders.

By Claim~\ref{clm:nhdgraphisramsey}, we know $G \rightarrow_q H$. Let $G' \subseteq G$ be a minimal subgraph that is still $q$-Ramsey for $H$. Claim~\ref{clm:nhdgraphisnotramsey} shows that we must have $w \in V(G')$, and in our application of Theorem~\ref{thm:arbitrarily_many_3-connected}, we shall take $v_0 = w$. The following claim, which we shall prove later, shows that $G'$ must contain at least one edge from the matching $M$.

\begin{clm} \label{clm:abundancematching}
The graph $G - M$ is not $q$-Ramsey for $H$.
\end{clm}

We thus have $e_i \in E(G')$ for some $i \in [q]$, and we take $e = e_i$ in Theorem~\ref{thm:arbitrarily_many_3-connected}. Given this preparation, it is simple to verify the conditions of the theorem. Indeed, we took $G'$ to be a minimal $q$-Ramsey graph. Moreover, recall that the neighbourhood of $w$ in $G$ is the vertex set of $\Gamma$. As $e_i$ is only connected to $\Gamma$ via signal senders, in which the distance between the signal edges is at least $v(H)$, it follows that there cannot be any copy of $H$ containing both $e_i$ and $w$. We can therefore apply Theorem~\ref{thm:arbitrarily_many_3-connected} to deduce the existence of minimal $q$-Ramsey graphs for $H$ with arbitrarily many vertices of degree $d_{G'}(w) \le d_G(w) = s_q(H)$, showing that $H$ is $s_q$-abundant.
\end{proof}

All that remains, then, is to prove Claim~\ref{clm:abundancematching}, a task we now complete.

\begin{proof}[Proof of Claim~\ref{clm:abundancematching}]
We need to exhibit an $H$-free colouring of $G - M$. This graph consists of three types of edges:
\begin{enumerate}
    \item those incident to $w$ or in the graph $\Gamma$,
    \item those in the subgraphs $R_{U,i}$ and between $R_{U,i}$ and $\Gamma$, for $U \in \binom{V(\Gamma)}{\delta(H)}$ and $i \in [q]$, and
    \item the edges within the signal senders.
\end{enumerate}

We colour all edges of (1) with the colour $1$, and all edges of (2) with the colour $2$. We finish by extending this colouring to an $H$-free $q$-colouring of each of the signal senders; note that this is possible, as each signal sender is missing at least one of its signal edges from $M$.

From Lemma~\ref{lem:safety}, we know that any copy of $H$ is either within a signal sender or outside it, and as we coloured the signal senders in an $H$-free fashion, it is only the colour-$1$ edges of (1) or the colour-$2$ edges of (2) that could give rise to a monochromatic copy of $H$.

We can rule out the former immediately. Either $e(\Gamma) = 0$, in which case the edges of (1) are simply a star around the vertex $w$, which cannot contain a copy of the well-behaved (and therefore $3$-connected) graph $H$, or $n > q(\delta(H) - 1) + 2 = v(\Gamma) + 1$, and so $\Gamma + \{w\}$ does not have enough vertices to support a copy of $H$.

To handle the latter case, observe that the argument in Claim~\ref{clm:nhdgraphisnotramsey} shows that no copy of $H$ can intersect two different subgraphs $R_{U,i}$ and $R_{U',i'}$, for $i, i' \in [q]$ and $U, U' \in \binom{V(\Gamma)}{\delta(H)}$. Hence, any copy of $H$ among the edges of (1) and (2) must use vertices of $R_{U,i}$ and $U$ for some $U \in \binom{V(\Gamma)}{\delta(H)}$ along with some vertex in $\Gamma - U$ or the vertex $w$. In either case these involve edges from (1) and therefore have the colour 1, and hence we cannot have a colour-$2$ copy of $H$.

This completes the proof of Claim~\ref{clm:abundancematching} and, with it, the proof of Proposition~\ref{prop:abundance}.
\end{proof}

\section{Non-simplicity for $G(n,p)$}\label{sec:non-simplicity}

In this section we prove the upper bounds on $\tilde{q}(H)$ from Theorem~\ref{thm:main}. These are the negative results, showing that with high probability $H \sim G(n,p)$ is not $q$-Ramsey simple for large values of $q$.

For the proofs, the centre of attention will again be the neighbourhood of the minimum degree vertex of $H$. We first prove upper bounds for case~\ref{thm:main:intermediateupperbounds} of our theorem. In the proof below, we first establish that, if the neighbourhood of the minimum degree vertex exhibits a high maximum degree, then $H$ cannot be $q$-Ramsey simple for a large enough $q$. For this, we will need the following result from~\cite{Kogan}. 

\begin{thm}[\cite{Kogan}]\label{thm:kindependence}
Let $G$ be an $n$-vertex graph of average degree $d$ and let $k \in \mathbb{N}$. Then there is a set $U$ of at least $(k+1)n/(d+k+1)$ vertices such that $\Delta(G[U]) \le k$.
\end{thm}

\begin{proof}[Proof of Theorem~\ref{thm:main}\ref{thm:main:intermediateupperbounds}]
Let $n^{-2/3}\ll p\ll 1$ and $H\sim G(n,p)$. By Lemma~\ref{lem:wellbehaved}, we know that $H$ has a unique vertex $u$ of minimum degree. As before, we set $F = H[N_H(u)]$.

Suppose that $H$ is $q$-Ramsey simple and $G$ is a minimal $q$-Ramsey graph for $H$ with minimum degree $N = q(\delta(H)-1)+1$. Let $w$ be a vertex of minimum degree in $G$, and $\Gamma = G[N_G(w)]$.
It follows from Remark~\ref{rem:nhdgraphnecessary} that there is an edge-colouring of $\Gamma$ such that the induced graph on every $\delta(H)$-set of vertices contains, in each colour, a copy of $F$. 

\smallskip
We are now ready to prove that $\tilde{q}(H) \le (1 + o(1))\frac{\delta(H)}{\Delta(F)}$.
The above observation implies that the induced graph on each $\delta(H)$-set has, in each colour, a vertex of degree at least $\Delta(F)$. However, the average degree of the sparsest colour class in $\Gamma$ is at most $\frac{2\binom{N}{2}}{qN}=\frac{N-1}{q} = \delta - 1$. Thus, by Theorem~\ref{thm:kindependence}, $\Gamma$ has a set of $\frac{\Delta(F)N}{\delta(H)+\Delta(F)-1}$ vertices
that induce a graph with maximum degree less than $\Delta(F)$ in this colour.
Hence, we must have
\[ \frac{\Delta(F)N}{\delta(H)+\Delta(F)-1} \le \delta(H) - 1, \]
which rearranges to give $q \le \frac{\delta(H)+\Delta(F)-1}{\Delta(F)} - \frac{1}{\delta(H)-1}$, from which the conclusion follows.
\medskip

We turn our attention to the second bound, namely $\tilde{q}(H) \le (1 + o(1)) \frac{\delta(H)^2}{2 e(F)}$. For any subset $U\subseteq V(\Gamma)$ of size $\delta(H)$, there must be a colour $i\in [q]$ such that there are at most $\frac{1}{q}\binom{\delta(H)}{2}$ edges of colour $i$ inside $U$. Using once again our observation above, we know that $\Gamma[U]$ contains a copy of $F$ in colour $i$, and therefore we must have $\frac{1}{q}\binom{\delta(H)}{2} \geq e(F)$, which yields the claimed bound. 
\end{proof}

We remark that the proofs of both upper bounds in Theorem~\ref{thm:main}\ref{thm:main:intermediateupperbounds} do not use the fact that $H$ is a random graph, and are valid for any graph that has a unique vertex of minimum degree whose neighbourhood is not an independent set. We end this section with a proof of Theorem~\ref{thm:main}\ref{thm:main:neversimple}.

\begin{proof}[Proof of Theorem~\ref{thm:main}\ref{thm:main:neversimple}]
Let $p\gg \sqrt{\frac{\log n}{n}}$ and $H\sim G(n,p)$. By Lemma~\ref{lem:qmonotonicity}, it suffices to show that 
a.a.s.~$H$ is not $q$-Ramsey simple for $q=2$.

For this, following Lemma~\ref{lem:gnp_triangles}, we may assume that every edge in $H$ belongs to a triangle. Now suppose for a contradiction that $H$ is $2$-Ramsey simple.
Let $G$ be a minimal $2$-Ramsey graph for $H$ such that
$G$ has a vertex $w$ with $d_G(w)=2\delta(H)-1$.
By the minimality of $G$, we can find an $H$-free
$2$-colouring $c$ of the graph $G-w$. 
Now fix an arbitrary vertex $v\in N_G(w)$ and observe that, by the pigeonhole principle, there must be a set  $W\subseteq N_G(w) \setminus\{v\}$
of size $\delta(H)-1$ such that all edges between $v$ and any of its neighbours in $W$ have the same colour; without loss of generality, let this be colour~$1$ and set $U=W\cup \{v\}$. We can extend the colouring $c$ by giving colour 1 to all edges in from $w$ to $N_G(w) \setminus U$ and giving colour 2 to all edges from $w$ to vertices in $U$. With this colouring we cannot create a monochromatic copy of $H$ in colour 1,
as $w$ is only incident to $\delta(H)-1$ edges of colour 1. 
On the other hand, $w$ is incident to exactly $\delta(H)$ edges of colour 2, which all lie between $w$ and $U$. Hence, if there were a monochromatic copy of $H$ in colour 2, the edge $wv$ would need to be part of it. However, since all edges in $U$ involving $v$ are of colour $1$, that means the edge $wv$ is not contained in any triangle of colour $2$, and hence cannot be in a monochromatic copy of $H$.
\end{proof}

\section{Concluding remarks and open problems}\label{sec:concluding}

In this paper we built upon the work of Grinshpun~\cite{grinshpun2015some} and studied the $q$-Ramsey simplicity of $H\sim G(n,p)$ for a wide range of values of $p$ and $q$. We encountered three different types of behaviour: for very sparse ranges, i.e.,~when $p\ll \frac{1}{n}$ or $\frac{\log n}{n} \ll p \ll n^{-\frac{2}{3}}$, we showed that a.a.s.~$H$ is $q$-Ramsey simple for every possible number of colours $q$;
for much denser ranges, i.e.,~when $p\gg \left(\frac{\log n}{n} \right)^{\frac{1}{2}}$,  a.a.s.~we do not have Ramsey simplicity even when $q=2$; in between these ranges, when $n^{-\frac{2}{3}}\ll p \ll n^{-\frac{1}{2}}$, there exists a finite threshold value $\tilde{q}(H) \ge 2$ on the number of colours $q$ such that $H$ is $q$-Ramsey simple if and only if $q\leq \tilde{q}(H)$. We determined this threshold up to a constant or, when $p = n^{-\frac12 - o(1)}$, logarithmic factor.  Several natural questions remain open.

\smallskip
First, our main result does not provide any information on the Ramsey simplicity of $G(n,p)$ when $p$ is between $\frac{1}{n}$ and $\frac{\log n}{n}$. 
\begin{question} 
What can be said about $\tilde{q}(H)$ when $H\sim G(n,p)$ and $\Omega \left( \frac{1}{n} \right) = p = O \left( \frac{\log n}{n} \right)$? In particular, is $H$ a.a.s.~$2$-Ramsey simple in this case?
\end{question}
In the range $p\gg \frac{\log n}{n}$ our simplicity proofs rely heavily on the fact that a.a.s.~$H\sim G(n,p)$ is 3-connected, implying the existence of signal senders for $H$, which in turn allow us to deduce a fairly general recipe for constructing suitable Ramsey graphs. When $p\ll \frac{1}{n}$, we know that $H\sim G(n,p)$ is a.a.s.~a forest, and simplicity follows from the construction of Szab\'o, Zumstein, and Z\"urcher~\cite{szabo2010minimum}, which works for certain bipartite graphs. When $\frac{1}{n} \ll p \ll \frac{\log n}{n}$, however, the random graph $G(n,p)$ becomes more complex (in particular, it is non-bipartite) but it is not yet connected. As a result, resolving the aforementioned question will likely require new ideas. 

\medskip

Second, in the range $\Omega \parens*{n^{-\frac{1}{2}}} = p = O \parens*{ \left(\frac{\log n}{n}\right)^{\frac{1}{2}}}$, we proved that $\tilde{q}(H)=O(p^{-1})$, which shows that the threshold 
value here is of smaller order than when $p = n^{-\frac12 - o(1)}$, as demonstrated in Corollary~\ref{cor:bounds}. However, we did not provide any nontrivial lower bounds, and we wonder if that might not be possible. 

\begin{question}
Is it true that $H$ is a.a.s.~not $2$-Ramsey simple when $H\sim G(n,p)$ with $\Omega \parens*{ n^{-\frac{1}{2}} } = p = O \parens*{ \left(\frac{\log n}{n}\right)^{\frac{1}{2}}} $?
\end{question}

In this case, signal senders for $H$ do exist, but the neighbourhood of the minimum degree vertex
becomes more complex than just a forest, making it difficult to construct a graph as described in Remark~\ref{rem:nhdgraphnecessary}. On the other hand, the presence of isolated vertices makes it likely that a more delicate argument than the one used in part~\ref{thm:main:neversimple} would be needed to show non-simplicity for smaller $q$. Nevertheless, we tend to believe that  a.a.s.~$\tilde{q}(H)=1$ 
for all $p\gg n^{-\frac{1}{2}}$.

\medskip
The bounds on $\tilde{q}(H)$ presented in cases \ref{cor:bounds:constanttreecomponents} and \ref{cor:bounds:smalldegrees} are already quite close, but it would be interesting to close the remaining gaps.
\begin{question}
Let $H\sim G(n,p)$ with $n^{-2/3} \ll p \ll n^{-1/2}$. What are the asymptotics of $\tilde{q}(H)$?
\end{question}
In this range, as we have seen in Section~\ref{sec:simplicity},
the question about $q$-Ramsey simplicity
is tightly linked to the problem of finding
a $q$-coloured graph $\Gamma$ on $q(\delta(H) - 1) + 1$ 
vertices such that the following holds:
For every set $U \subseteq V(\Gamma)$ of $\delta(H)$ vertices and for every colour $i \in [q]$, there exists a copy $F_{U,i}$ of $F=H[N(u)]$ in $\Gamma[U]$ whose edges are all of colour $i$.
The proofs of our lower bounds in Section~\ref{sec:simplicity} are obtained by finding such $\Gamma$ (with additional properties as given in Proposition~\ref{prop:nhdgraph})
through explicit constructions or probabilistic arguments.
In order to prove that a.a.s.~$H$ is not $q$-Ramsey simple,
it would suffice to prove that such $\Gamma$ does not exist, that is, every $q$-coloured graph $\Gamma$ on $q(\delta(H) - 1) + 1$  contains \textit{at least} one subset $U \subseteq V(\Gamma)$ of size $\delta(H)$ such that $\Gamma[U]$ is missing a copy of $F$ in at least one colour.
Note that in the proof of our second bound in 
case~\ref{thm:main:intermediateupperbounds} of Theorem~\ref{thm:main}
we obtain such a result by a simple counting argument
which guarantees that we cannot pack $q$ copies of $F$ into any graph on $\delta(H)$ vertices. Related to this argument, it seems challenging to determine how many copies of a given random graph can be packed into a complete graph, leading us to suggest the following question. 

\begin{question}
Let $H\sim G(n,p)$ with $0<p<1$. How many copies of $H$
can be packed into $K_n$?
\end{question}
\medskip

In the densest range, that is, when $p\gg \left(\frac{\log n}{n}\right)^{\frac{1}{2}}$, we know that $H\sim G(n,p)$ is a.a.s.~not $q$-Ramsey simple for any $q\geq 2$. We wonder, however, what the behaviour of $s_q(H)$ in this case is; in particular, it would be interesting to determine whether $s_q(H)$ is still typically close to the easy lower bound $q(\delta(H)-1)+1$.
Note that the answer is no if $p=1$ and $q \ge 2$, since $s_2(K_n)=(n-1)^2$. However, when $\left(\frac{\log n}{n}\right)^{\frac{1}{2}}\ll p \ll 1$, we do not know of any bounds other than the general ones mentioned in the introduction. 
In particular, we propose the following problem, similar to one posed by Grinshpun, Raina, and Sengupta~\cite{grinshpun2017minimum}.

\begin{question}
How large is $s_2(H)$ for $H\sim G(n,\frac{1}{2})$ a.a.s.?
\end{question}
\medskip

Related to the above discussion, we also note that our methods can be applied to the 2-colour asymmetric Ramsey setting, in which a graph $G$ is said to be $2$-Ramsey for a pair of graphs $(H_1,H_2)$ if every red-/blue-colouring of its edges leads to a red copy of $H_1$ or a blue copy of $H_2$. In this setting,
we define minimal Ramsey graphs and the smallest minimum degree $s_2(H_1,H_2)$ in the obvious way; the general lower bound is replaced by
$s_2(H_1,H_2)\geq \delta(H_1)+\delta(H_2)-1$ and
again we call a pair $(H_1,H_2)$ 2-Ramsey simple if this lower bound is attained.
Our constructions can be modified to show that for $H_1\sim G(n,p_1)$ and $H_2\sim G(n,p_2)$ the pair
$(H_1,H_2)$ is a.a.s. 2-Ramsey simple if $\frac{\log n}{n} \ll p_1\leq p_2 \ll n^{-1/2}$. When $p_1,p_2\ll n^{-1}$, then again a modification of the argument of Szab\'o, Zumstein, and Z\"urcher~\cite{szabo2010minimum}
can be used to show that we a.a.s. have 2-Ramsey simplicity. Still, the following questions remain.

\begin{question}
Let $H_1\sim G(n,p_1)$ and $H_2\sim G(n,p_2)$
with $p_1\ll n^{-1}$ and $\frac{\log n}{n} \ll p_2 \ll n^{-1/2}$. Is the pair $(H_1,H_2)$ a.a.s. 2-Ramsey simple? What happens if one of the graphs comes from the dense range?
\end{question}

We also remark that our ideas from Section~\ref{sec:simabundance} can be used to resolve a special case of a conjecture due to Grinshpun~\cite{grinshpun2015some}, stating that all triangle-free graphs are $2$-Ramsey simple. In~\cite{grinshpun2017minimum}, Grinshpun, Raina, and Sengupta use a construction similar to ours to show that the conjecture is true for all regular 3-connected triangle-free graphs satisfying one extra technical condition. Our approach allows us to prove that every well-behaved triangle-free graph is $q$-Ramsey simple for any $q\geq 2$.

\medskip

Finally, let us emphasise that there has been little study of (minimal) Ramsey graphs for $G(n,p)$. The only results we are aware of concern the Ramsey number of $G(n,p)$, as mentioned in Section~\ref{intro:random.graphs}. Hence, as a more general direction for future research, it would be interesting to explore other aspects of the Ramsey behaviour of $G(n,p)$ as the target graph.

\paragraph{Acknowledgements} The first author was supported by the Deutsche Forschungsgemeinschaft Graduiertenkolleg ``Facets of Complexity'' (GRK 2434). The third author was supported by the Deutsche Forschungsgemeinschaft project 415310276.

\bibliographystyle{amsplain}
\bibliography{biblio}

\appendix
\section{Forests are Ramsey simple}\label{sec:forests}

\begin{lemma}\label{lem:forest}
For every forest $F$ without isolated vertices and every integer $q\geq 2$, we have $s_q(F)=1$. 
\end{lemma}

For two colours, Lemma~\ref{lem:forest} follows from the general result for bipartite graphs of Szab\'o, Zumstein, and Z{\"u}rcher~\cite[Theorem~1.3 and Corollary~1.5]{szabo2010minimum}.
Their proof generalises easily to more colours. For the sake of completeness, we include a simplified version of that proof here, dealing only with the case of forests.

\begin{proof}
Given $F$, fix a bipartition $V(F)=A\cup B$, where $|A|\leq |B|$ and the size of $A$ is minimised. Set $a=|A|$, $b=|B|$, $B_1=\{v\in B:~ d_F(v)=1\}$, and 
$B_{\geq 2}=B\setminus B_1$. We start by showing that $|B_{\geq 2}|\leq a-1$. Indeed, let $T_1,\ldots,T_k$ be the components of $F$ and, for each $i\in [k]$, let $r_i$ be an arbitrary vertex in $A\cap V(T_i)$. Viewing $T_i$ as a tree rooted at $r_i$, we note that each element of $B_{\geq 2}\cap V(T_i)$ must have a child in $A\cap V(T_i)$ and, since $T_i$ is a tree, all of these children must be different. Thus, $|A\cap V(T_i)|\geq |B_{\geq 2}\cap V(T_i)|+1$ for all $i\in [k]$, and summing up over all components yields $|A|\geq |B_{\geq 2}|+1$.

Now, set $r=q(a-1), s=q^{r+1}v(F),$ and $t=sbq$, and let $G$ be the graph constructed as follows: 
\begin{itemize}
\item let $V(G)=X\dot{\cup} Y\dot{\cup} Z$, where $\card{X} = r$, $\card{Y} = s$, and $\card{Z} = t$,
\item add a complete bipartite graph between $X$ and $Y$, and
\item partition $Z$ into $s$ subsets of size $bq$, indexed by the elements of $Y$. That is, let $Z = \bigcup\limits_{y\in Y}Z_y$, where $\card{Z_y} = bq$. For each $y\in Y$, connect $y$ to all vertices of $Z_y$.
\end{itemize}

Each vertex $v\in Z$ then satisfies $d_G(v)=1$. We will now show that (i) $G - Z \not\rightarrow_q F$,
and (ii) $G \rightarrow_q F$.
From this it follows directly that $G$ must contain a graph from $\mathcal{M}_q(F)$ with minimum degree one, and hence that  $s_q(F)=1$.

\medskip

To see property (i), colour $E(G-Z)$ as follows:
take any partition $X=X_1\cup\ldots \cup X_q$ with $|X_i|=a-1$
for every $i\in [q]$, and colour $E(X_i,Y)$ in colour $i$.
Then each colour class is a bipartite graph with a partite set of size smaller than $a$. By the definition of $a$, there cannot be a monochromatic copy of $F$.

\medskip

We prove (ii) next. Consider any $q$-colouring
$\varphi: E(G)\rightarrow [q]$. Each vertex
$y\in Y$ has $bq$ neighbours in $Z_y$, and hence there must be
a subset $Z_y'\subseteq Z_y$ of size $b$ such that the edges from $y$ to $Z_y'$ are monochromatic. 
As we use only $q$ colours, there must be a subset $Y'\subseteq Y$ of $\frac{s}{q}$ 
vertices $y_1,\ldots,y_{s/q}$
such that, without loss of generality, the edges between $y_i$ and $Z_{y_i}'$ are all colour $1$ for each $i \in [s/q]$. Further, set 
$Z'=\bigcup\limits_{y_i\in Y'} Z_{y_i}'$.
Next, let $X=\{x_1,\ldots,x_r\}$. For every $y_i\in Y'$, we consider the vector $c_i:=(\varphi(x_1y_i),\ldots,\varphi(x_ry_i))\in [q]^r$, the \emph{colour profile} of $y_i$. As $|Y'|=s/q = q^r v(F)$, there must be at least $v(F)$ vertices in $Y'$ with the same colour profile $c$.
By symmetry, we may assume that $c_1=c_2=\ldots=c_{v(F)}=c$. We consider two cases. 

\medskip

{\bf Case 1:} There is a colour $i\in [q]$ that appears
at least $a$ times in $c$. This gives a copy of $K_{a,v(F)}$
between $X$ and $\{y_1,\ldots,y_{v(F)}\}$
that is monochromatic in colour $i$. As $K_{a,v(F)}$ contains a copy of $F$, we are done.

\medskip

{\bf Case 2:} Every colour is used exactly $(a-1)$ times in $c$.
In particular, we find a subset $X'\subseteq X$ of size $a-1$ such that the edges between $X'$ and $\{y_1,\ldots,y_{v(F)}\}$ are monochromatic in colour $1$. Using the edges between $y_i$ and $Z_{y_i}'$, for $i\in [v(F)]$, we find a monochromatic copy of $F$: embed $A$ into $\{y_1,\ldots,y_{v(F)}\}$ arbitrarily, embed $B_1$ into $Z'$ by respecting adjacency relation, and embed $B_{\geq 2}$ into $X'$ arbitrarily.
\end{proof}

\end{document}